%% file: DYS.tex
\documentclass[hidelinks]{siamart220329}

\input{./macros}
\ifpdf
\hypersetup{
  pdftitle={Convergence Analysis of DYS via SRG},
  pdfauthor={Lee, Yi, and Ryu}
}
\fi

\usepackage{booktabs}

\newtheorem{fact}{Fact}

\makeatletter
\newcommand{\customfootnote}[2]{%
  \xdef\@thefnmark{}\@footnotetext{$\!\!\!$\textsuperscript{#1}#2}
}
\makeatother

\begin{document}

\maketitle

\customfootnote{$\dagger$}{Department of Mathematical Sciences, Seoul National University.}
\customfootnote{$\ddagger$}{Equal contribution.}
\customfootnote{$\S$}{Department of Statistics and Data Science, Carnegie Mellon University.}

\begin{abstract}
Davis--Yin splitting (DYS) has found a wide range of applications in optimization, but its linear rates of convergence have not been studied extensively. The scaled relative graph (SRG) simplifies the convergence analysis of operator splitting methods by mapping the action of the operator to the complex plane, but the prior SRG theory did not fully apply to the DYS operator. In this work, we formalize an SRG theory for the DYS operator and use it to obtain tighter contraction factors.
\end{abstract}



\section{Introduction}
The Davis--Yin splitting (DYS) operator \cite{davis_three-operator_2017} encodes solutions of the three-operator monotone inclusion problem
	\begin{align*}
		\underset{x \in \cH}{\text{find}} \quad 0 \in (\opA + \opB + \opC)x,
	\end{align*}
	where $\cH$ is a real Hilbert space, $\opA$ and $\opB$ are monotone, and $\opC$ is cocoercive,
	as fixed-points of the operator 
	\[
		\opT = \opI - \resa{\opB} + \resa{\opA}(2\resa{\opB} - \opI - \alpha \opC \resa{\opB}),
	\]
	where $\alpha>0$ and  $\resa{\opA}$ and $\resa{\opB}$ are the resolvents of $\opA$ and $\opB$.
	DYS generalizes the forward-backward splitting (FBS) and Douglas--Rachford splitting (DRS) operators, and it has found a wide range of applications \cite{yan2018new, wang2017three, carrillo2022primal, van2020three, weylandt2019splitting, heaton2021learn} in optimization. However, linear convergence rates of DYS iterations have not been studied extensively.
	
	
On the other hand, the recently proposed scaled relative graph (SRG) \cite{ryu_scaled_2021} allows one to analyze the convergence of operator splitting methods by mapping the action of a multi-valued nonlinear operator to the extended complex plane, analogous to how the spectrum maps the action of a linear operator to the complex plane. Prior work used the SRG to analyze rates of convergence for FBS and DRS \cite{ryu_scaled_2021}, but the prior SRG theory did not fully apply to the DYS operator.

In this paper, we analyze convergence rates of DYS iterations using the SRG. First, we extend the SRG theory to accommodate the DYS operator. We then establish linear rates of convergence of DYS by bounding the maximum modulus of the SRG on the complex plane. We furthermore provide a new result on averagedness coefficients of DYS.	



\subsection{Prior Works}
	
	
	Splitting methods for monotone inclusion problems is a powerful framework encompassing a wide range of convex optimization algorithms \cite{bauschke_2017_convex, ryu_primer_2016, ryu2021large}. Classical splitting methods such as forward-backward splitting (FBS) \cite{bruck1977weak, passty1979ergodic, beck2009fast}, Douglas--Rachford splitting (DRS) \cite{peaceman1955numerical, douglas1956numerical, lions1979splitting}, and alternating directions method of multipliers (ADMM) \cite{gabay1976dual, Glowinski1975}, a derived splitting algorithm, have had an enormous range of applications. As a splitting scheme unifying FBS and DRS, Generalized forward-backward splitting (GFBS) \cite{raguet2013generalized} was first proposed, and Forward-Douglas--Rachford splitting (FDRS) \cite{briceno2015forward} was then presented as an extension of GFBS. In 2017, Davis and Yin proposed what we call the Davis-Yin splitting (DYS) as a complete unification of FBS and DRS \cite{davis_three-operator_2017}.
    This breakthrough has found a wide range of new applications \cite{yan2018new, wang2017three, carrillo2022primal, van2020three, weylandt2019splitting, heaton2021learn} and many variants of DYS such as stochastic DYS \cite{yurtsever_three_2021, yurtsever2016stochastic, cevher2018stochastic, yurtsever_three_2021-1, pedregosa2019proximal}, inexact DYS \cite{zong2018convergence}, adaptive DYS  \cite{pedregosa_adaptive_2018}, inertial DYS \cite{Cui2019AnIT}, and primal-dual DYS\cite{salim2022dualize} have been proposed. 
	
	However, the study of linear convergence rates of DYS has been limited. Bredies, Chenchene, Lorenz, and Naldi studied the convergence of DYS under a wider range of stepsizes \cite{bredies2022degenerate} and Arag{\'o}n-Artacho and Torregrosa-Bel{\'e}n also proved the convergence of the DYS iteration over a wider range of stepsizes \cite{ Aragon-ArtachoTorregrosa-Belen2022_direct}. Dao and Phan analyzed a modified version of DYS iterations \cite{DaoPhan2021_adaptive}, and Pedregosa analyzed sublinear convergence rates of DYS iterations  \cite{pedregosa_convergence_2016}, but they did not analyze linear rates of convergence. Condat and Richt{\'a}rik analyzed linear convergence rates of a randomized version of Primal-Dual Davis--Yin (PDDY) \cite{condat2022randprox}. Ryu, Taylor, Bergeling, and Giselsson \cite{ryu_operator_2020} and Wang, Fazlyab, Chen, and Preciado \cite{wang_robust_2019} formulated the problem of computing tight contraction factors as SDPs using the performance estimation problem (PEP) and integral quadratic constraint (IQC) approaches, but these are numerical results, not analytic convergence proofs in the classical sense. 
    To the best of our knowledge, the only prior works providing analytic linear convergence rates of DYS iterations are the arXiv version of the original DYS paper \cite{davis_three-operator_2017}  and \cite{condat2022randprox}. The latter work \cite{condat2022randprox} considers randomized updates, which makes the setup distinct from ours.


	On the other hand, the theory of scaled relative graph (SRG) by Ryu, Hannah, and Yin \cite{ryu_scaled_2021} allows one to analyze the convergence of operator splitting methods by mapping the action of a multi-valued nonlinear operator to the extended complex plane, analogous to how the spectrum maps the action of a linear operator to the complex plane. Follow-up work has extended the theory and applied it to analyze nonlinear operators: Huang, Ryu, and Yin characterized the SRG of normal matrices \cite{huang_scaled_2019}, Pates further characterized the SRG of linear operators using the Toeplitz--Hausdorff theorem \cite{pates_scaled_2021}, and Huang, Ryu, and Yin \cite{huang_tight_2020} established the tightness of the averagedness coefficient of the composition of averaged operators \cite{ogura_non-strictly_2002} and the DYS operator \cite{davis_three-operator_2017}. SRG has also found applications in control theory: Chaffey, Forni, and Rodolphe utilized the SRG to analyze input-output properties of feedback systems \cite{chaffey_graphical_2021,chaffey_scaled_2021}, and Chaffey and Sepulchre furthermore used it as an experimental tool to determine properties of a given model\cite{chaffey_monotone_2021,chaffey_rolled-off_2021, chaffey2022circuit}.
	

	\subsection{Contribution and organization}

Section~\ref{s:srg_theory} presents our first major contribution, formalizing an SRG theory for the DYS operator. This result reduces the problem of analyzing the DYS operator into bounding the maximum modulus of a subset of the complex plane. Section~\ref{s:dys_cont} presents our second major contribution, analyzing contraction factors of the DYS operator. These factors are tighter than those shown in \cite{davis_three-operator_2017} and are, to the best of our knowledge, the best known factors. Finally, Section~\ref{s:dys_avg} presents a minor contribution, a tight averagedness coefficient of the DYS operator under the assumption of strong monotonicity of $\opA$ or $\opB$ and Lipschitzness of $\opC$. This is, to the best of our knowledge, the first averagedness result of
DYS without cocoercivity of $\opC$.

\subsection{Preliminaries}

	Write $\cH$ for a real Hilbert space with inner product $\dotp{\cdot, \cdot}$ and norm $\norm{\cdot}$. 
Write $\ecomplex=\complex\cup\{\infty\}$ to denote the extended complex plane, and define $1/0=\infty$ and $1/\infty=0$ in $\ecomplex$.
	We use Minkowski-type set notations: for $\alpha \in \complex$ and $W, Z \subseteq \complex$, write
	\begin{gather*}
		W \alpha = \alpha W = \{\alpha w \,|\, w \in W\}, \quad W + \alpha = \alpha + W = \{\alpha + w \,|\, w \in W\},
	\end{gather*}
	and 
	\[
		W+Z = \{w+z\,|\,w \in W, z \in Z\}, \quad WZ = \left\{ wz \,|\, w \in W,\, z \in Z \right\}.
	\]
	For $W \subset \complex$, write
	\[
		\abs{W} = \left\{ \abs{w} \,|\, w \in W \right\}, 
	\]
	where $\abs{w}$ is absolute value (modulus) of a complex number $w$.
    To clarify,  $\abs{W}$ does not refer to the cardinality of $W$.
	For $W \subset \ecomplex$, write
	\[
			\quad 
		W^{-1} = \left\{ w^{-1} \,|\, w \in W \right\}.
\]
    For $z \in \complex$, write $\Re z$ and $\Im z$ for the real and imaginary part of $z$, respectively, and $\arg z$ for the argument of $z$ that lies within $(-\pi, \pi]$.

	\paragraph{\textbf{Geometry}}
	Given any $x, y\in\cH$, define 
	\begin{align*}
		\angle(x,y) =
		\left\{
		\begin{array}{ll}
		\arccos\para{\tfrac{\dotp{x,y}}{\norm{x}\norm{y}}}&\text{ if }x\ne 0,\,y\ne 0\\
		0&\text{ otherwise}
		\end{array}
		\right.
	\end{align*}
	to be the angle between $x$ and $y$.
	Note that $\angle(x, y)\in [0,\pi]$.
	The following spherical triangle inequality is the analog of the triangle inequality on spheres.
	\begin{fact}[Fact 17 \cite{ryu_scaled_2021}]
		\label{fact:sphere-tri-ineq}
		Let $\cH$ be a real Hilbert space.
		Then, 
		\[
			\left| \angle(x, y) - \angle(y, u) \right|
			\le \angle(x, u) 
			\le \angle(x, y) + \angle(y, u),\qquad\forall \,x,y,u\in \cH.
		\]
	\end{fact}
	We define notations for disks and circles on the complex plane.
	For $z \in \complex$ and $r \in (0, \infty)$, define
	\[
	    \Ci(z, r) = \{ w \in \complex \,|\, \abs{w - z} = r\},
	    \qquad
	    \Di(z, r) = \{ w \in \complex \,|\, \abs{w - z} \le r \}.
	\]
    
	\paragraph{\textbf{Multi-valued operators}} In general, we follow standard convex operator theoretical notations as in \cite{bauschke_2017_convex,ryu2021large}.
    Write $\opA \colon \cH \rightrightarrows \cH$ to denote that $\opA$ is a multi-valued operator.
    Define $\dom\opA = \left\{ x \in \cH \,|\, \opA x \neq \emptyset \right\}$.
    If all outputs of $\opA$ are singletons or the empty set, we say $\opA$ is single-valued.
    In this case, we identify $\opA$ with the function from $\dom\opA$ to $\cH$.
	Define the graph of an operator $\opA$ as
	\[
		\mathrm{graph}(\opA) = \{(x, u) \in \cH \times \cH \,|\, u \in \opA x\}.
	\]
	For notational simplicity, we often do not distinguish an operator from its graph and, for example, write $(x,u) \in \opA$ to mean $u \in \opA x$.
	Define the inverse of $\opA$ as
	\[
		\opA ^{-1} = \{(u, x) \,|\, (x,u) \in \opA\},
	\]
	scalar multiplication with an operator as
	\[
		\alpha\opA  = \{(x,\alpha u) \,|\, (x, u) \in \opA\},
	\]
	the identity operator as
	\[
	    \opI = \{ (x, x) \,|\, x \in \cH \},
	\]
        and 
        \[
            \opA + \alpha\opI = \alpha\opI + \opA = \{ (x, \alpha x + u) \,|\, (x, u) \in \opA \}
        \]
        for any $\alpha\in\reals$.
	Define the resolvent of $\opA$ with stepsize $\alpha>0$ as 
	\[
	    \opJ_{\alpha \opA} = (\opI +\alpha \opA)^{-1}.
	\]
 Define addition and composition of operators $\opA \colon \cH \rightrightarrows \cH$ and $\opB \colon \cH \rightrightarrows \cH$ as
	\begin{align*}
		\opA + \opB &= \left\{ (x, u + v) \,|\, (x, u) \in \opA, (x, v) \in \opB \right\},\\
		\opA \opB &= \left\{ (x, s) \,|\, \exists \, u \; \text{such that} \; (x, u) \in \opB, (u, s) \in \opA \right\}.
	\end{align*}

	We call $\cA$ a class of operators if it is a set of operators. 
	Operators in a same class need not be defined on the same Hilbert space.
	For an operator $\opA \colon \cH \rightrightarrows \cH$ and an operator class $\cB$,
	we define sum and product of them as 
	\begin{align*}
	    \opA + \cB &= \{ \opA + \opB \,|\, \opB \in \cB, \opB \colon \cH \rightrightarrows \cH \}, \\
	    \opA\cB &= \{ \opA\opB \,|\, \opB \in \cB, \opB \colon \cH \rightrightarrows \cH \}
	\end{align*}
	Note that we restrict operands to be defined on the same domain Hilbert space $\cH$.
	For an operator class $\cA$ and $\opB \colon \cH \rightrightarrows \cH$, define $\cA + \opB$ and $\cA\opB$ in the same way.
	We define sum of classes of operators $\cA$ and $\cB$ as
	\begin{align*}
	    \cA + \cB = \{\opA + \opB \,|\, &\opA \in \cA, \,\opB \in \cB, \opA \colon \cH \rightrightarrows \cH,\, \opB \colon \cH \rightrightarrows \cH,\\
            & \cH\text{ is a Hilbert space}\}.
	\end{align*}
	We define $\cA\cB$ similarly.
        For any $\alpha\in \reals$, define
        \[
            \alpha \cA = \{ \alpha \opA \,|\, \opA \in \cA\}
        \]
        and \[
	   \cA + \alpha\opI = \alpha\opI + \cA = \{ \alpha\opI + \opA \,|\, \opA \in \cA\}.
	\]
	Define 
	\[
		\cA^{-1} = \{\opA^{-1} \,|\, \opA \in \cA \}
	\]
	and $\opJ_{\alpha\cA} = (\opI + \alpha\cA)^{-1}$ for $\alpha > 0$.
	
	We now define some special operator classes.
	Define the class of monotone operators as
	\[
		\cM = \{ \opA \colon \cH \rightrightarrows \cH \,|\, \dotp{x-y, u-v} \geq 0, \forall\, u \in \opA x, v \in \opA y ,\,\cH\text{ is a Hilbert space}\}.
	\]
	For $\mu \in (0, \infty)$, define the class of $\mu$-strongly monotone operators as
        \begin{align*}
        \sm\mu = \{\opA \colon \cH \rightrightarrows \cH \,|\, &\dotp{x-y, u-v} \geq \mu \norm{x-y}^2, \,\forall\, u \in \opA x, v \in \opA y,\\
        &\cH\text{ is a Hilbert space}\}.    
        \end{align*}
	For $L \in (0, \infty)$, define the class of $L$-Lipschitz operators as
        \begin{align*}
            \cL_L = \{ \opA \colon \dom \opA \to \cH \,|\, &\norm{ u-v}^2 \leq L^2\norm{x-y}^2, \forall\, u \in \opA x, v \in \opA y, \\
            &\dom \opA \subseteq\cH,\,\cH\text{ is a Hilbert space}\}.
        \end{align*}
	For $\beta \in (0, \infty)$, define the class of $\beta$-cocoercive operators as
        \begin{align*}
            \cC_\beta = \{ \opA \colon \dom \opA \to \cH \,|\, &\dotp{x-y, u-v} \geq \beta \norm{u-v}^2, \forall\, u \in \opA x, v \in \opA y, \\
            &\dom \opA \subseteq\cH,\,\cH\text{ is a Hilbert space}\}.
        \end{align*}
	For $\theta \in (0, 1)$, define the class of $\theta$-averaged operators as 
	\[
		\cN_\theta 	= (1-\theta)\opI + \theta \cL_1.
	\]
 
	Note that we automatically identify Lipschitz and cocoercive operators, as well as resolvents of monotone operators, as singled-value functions since they are necessarily single-valued.
	For notational convenience, we write
	\[
		\sm{0} = \coco{0} = \cM, 
		\quad 
		\lip{\infty} = \left\{ \opA \colon \cH \rightrightarrows \cH  \,|\, \cH\text{ is a Hilbert space}\right\}.
	\]
	To clarify,
	\[
		\sm{0} \ne \bigcup_{\mu > 0} \sm{\mu},
		\quad
		\coco{0} \ne \bigcup_{\beta > 0} \coco{\beta},
		\quad
		\lip{\infty} \ne \bigcup_{L < \infty} \lip{L}.
	\]

	\paragraph{\textbf{Scaled relative graphs}} 
	Define the SRG of an operator $\opA\colon\cH\rightrightarrows\cH$ as
	\begin{align*}
		\mathcal{G}(\opA)&=
		\left\{
		\frac{\|u-v\|}{\|x-y\|}
		\exp\left[\pm i \angle (u-v,x-y)\right]
		\,\Big|\,
		u\in \opA x,\,v\in \opA y,\, x\ne y \right\}\\
		&\qquad\qquad\qquad\qquad\qquad\qquad\qquad\qquad\qquad
		\bigg(\cup \{\infty\}\text{ if $\opA$ is not single-valued}\bigg).
	\end{align*}

SRG depicts a change in output relative to the change in input of operators using complex numbers, analogous to how eigenvalues capture essential information about the action of linear operators.
	We define the SRG of a class of operators $\cA$ as
	\[
	\mathcal{G}(\cA)=\bigcup_{\opA\in \cA}\mathcal{G}(\opA).
	\]

	
	
	\begin{fact}[Proposition 1 \cite{ryu_scaled_2021}]
	\label{fact:monotone-srg}
	Let $\mu$,$\beta$,$L \in (0,\infty)$ and
         $\theta \in (0, 1)$.
	The SRGs of $\cL_L$, $\cN_\theta$, $\cM$, $\sm\mu$, and $\cC_\beta$ are, respectively, given by
	\begin{center}
	\begin{tabular}{cc}
	\raisebox{-.5\height}{
	\def\scale{1.4}
	\begin{tikzpicture}[scale=\scale]
	\draw (-1.1,0.5) node {$\cG(\cL_L)$};
	\fill[fill=medgrey] (0,0) circle (0.75);
	\draw [<->] (-1,0) -- (1,0);
	\draw [<->] (0,-1) -- (0,1);
	\draw (0.75,0) node [above right] {$L$};
	\draw (-0.75,0) node [above left] {$-L$};
	\filldraw (.75,0) circle ({\ptsize / \scale pt});
	\filldraw (-.75,0) circle ({\ptsize / \scale pt});
	\draw (0,-1.2) node  {$\Di(0, L)$};
	\end{tikzpicture}}
	&
	\raisebox{-.5\height}{
	\def\scale{1.4}
	\begin{tikzpicture}[scale=\scale]
	\draw (-0.85,0.5) node {$\cG(\cN_\theta)$};
	\fill[fill=medgrey] (0.25,0) circle (0.75);
	\draw [<->] (-.8,0) -- (1.2,0);
	\draw [<->] (0,-1) -- (0,1);
	\draw (1,-0) node [above right] {$1$};
	\draw (-0.5,-0) node [below] {$1-2\theta$};
	\filldraw (1,0) circle ({\ptsize / \scale pt});
	\filldraw (.25,0) circle ({\ptsize / \scale pt});
	\filldraw (-.5,0) circle ({\ptsize / \scale pt});
	\draw (0.625,0.1) node [above] {$\theta$};
	\draw [decorate,decoration={brace,amplitude=4.5pt}] (0.25,0) -- (1,0) ;
	\draw (0.4,-1.2) node  {$\Di(1 - \theta, \theta)$};
	\end{tikzpicture}}
	\end{tabular}
	\end{center}
	\begin{center}
	\begin{tabular}{ccc}
	\raisebox{-.5\height}{
	\def\scale{1.4}
	\begin{tikzpicture}[scale=\scale]
	\draw (-0.5,0.5) node {$\cG(\cM)$};
	\fill[fill=medgrey] (0,-1) rectangle (1,1);
	\draw [<->] (-0.2,0) -- (1,0);
	\draw [<->] (0,-1) -- (0,1);
	\draw (0.7,.85) node {$\cup \{\binfty\}$};
	\draw (0.4,-1.2) node {$\left\{z\in \complex\,|\,\Re z\ge 0\right\}\cup\{\infty\}$};
	\end{tikzpicture}}
	&
	\!\!\!\!\!\!
	\raisebox{-.5\height}{
	\def\scale{1.4}
	\begin{tikzpicture}[scale=\scale]
	\draw [<->] (0,-1) -- (0,1);
	
	\draw (-0.2,0.5) node [fill=white] {$\cG(\sm\mu)$};
	\fill[fill=medgrey] (0.3,-1) rectangle (1,1);
	
	\draw [<->] (-0.2,0) -- (1,0);
	
	\draw (0.7,.85) node {$\cup \{\binfty\}$};
	\filldraw (.3,0) circle ({\ptsize / \scale pt});
	\draw (0.3,0) node [above left] {$\mu$};
	\draw (0.4,-1.2) node {$\left\{z\in \complex\,|\,\Re z\ge \mu\right\}\cup\{\infty\}$};
	
	\end{tikzpicture}}
	&
	\!\!\!\!\!\!
	\raisebox{-.5\height}{
	\def\scale{1.4}
	\begin{tikzpicture}[scale=\scale]
	\draw (-0.4,0.3) node {$\cG(\cC_\beta)$};
	\fill[fill=medgrey] (0.6,0) circle (0.6);
	\draw [<->] (-0.5,0) -- (1.5,0);
	\draw [<->] (0,-.8) -- (0,.8);
	\filldraw (1.2,0) circle ({\ptsize / \scale pt});
	\draw (1.2,0) node [above right] {\!$1/\beta$};
	\draw (0.4,-1) node  {$\Di\left(\frac{1}{2\beta}, \frac{1}{2\beta}\right)$};
	\end{tikzpicture}}
	\end{tabular}
	\end{center}
	\end{fact} 

        We say a class of operators $\cA$ is \emph{SRG-full} if
	\begin{align*}
	\opA\in \cA
	\quad\Leftrightarrow\quad
	\cG(\opA)\subseteq\cG(\cA).
	\end{align*}		%
In essence, a class is SRG-full when it can be entirely characterized by its SRG.
	\begin{fact}[Theorem 2, 3 \cite{ryu_scaled_2021}]
	\label{fact:srg-full}
	Let $\mu$,$\beta$,$L \in (0,\infty)$ and
         $\theta \in (0, 1)$.
	Then, $\cM$, $\sm\mu$, $\cC_\beta$, $\cL_L$, and $\cN_\theta$ are SRG-full classes. 
	In addition, their intersections are also SRG-full classes.
	\end{fact}

    \begin{fact}[Theorem 4, 5 \cite{ryu_scaled_2021}]
    \label{fact:srg-trans}
         If $\cA$ is a class of operators, then 
	\[
	\cG(\alpha \cA)=\alpha\cG(\cA), \quad
	\cG(\opI+ \cA)=1+\cG(\cA), \quad
	\cG(\cA^{-1}) = \cG(\cA)^{-1}.
	\]
	where $\alpha$ is a nonzero real number.
	If $\cA$ is furthermore SRG-full, then $\alpha \cA, \opI+ \cA$, and $\cA^{-1}$
	are SRG-full.
    \end{fact}
	Note that Fact~\ref{fact:srg-trans} implies that
	$\opJ_{\alpha\cA} = (\opI + \alpha\cA)^{-1}$ is an SRG-full class  for $\alpha > 0$ provided that $\cA$ is an SRG-full class.
	
	
	\begin{fact}\label{fact:interpolable}
	If $\cA$ is an SRG-full class and 
	\[
		\frac{\norm{u - v}}{\norm{x - y}}\exp[i\angle(u - v, x - y)] \in \srg{\cA}
	\]
	for any	$u$, $v$, $x$, $y \in \cH$ with $x \neq y$, there exists a single-valued operator $\opA \in \cA$ such that
	\[
		\opA x = u,\quad \opA y = v. 
	\] 
	\end{fact}
	\begin{proof}
	Define $\opA=\{(x,u),(y,v)\}$.
	Then, 
	\[
	    \srg{\opA} = 
	    \frac{\norm{u - v}}{\norm{x - y}}
	    \exp[\pm i \angle(u - v, x - y)] \subseteq \srg{\cA}
	\]
	holds, so $\opA \in \cA$ by SRG-fullness.
	\end{proof}

        \begin{fact}\label{fact:single-pt-op}
            Let $\cH$ be a real Hilbert space.
            If $\opA \colon \cH \rightrightarrows \cH$ is a single-valued operator such that $\dom\opA$ is a singleton,
            $\opA$ is an element of any SRG-full class.
        \end{fact}
        \begin{proof}
            $\srg{\opA} = \emptyset$ is a subset of the SRG of any SRG-full class.
        \end{proof}

	\paragraph{\textbf{Arc property}}
    Denote $\conj{z}$ the conjugate of $z \in \complex$.
	Define the right-hand arc between $z, \conj{z} \in \complex$ as
	\[
		\rarc(z, \conj{z})	=  
		\left\{
			r e^{i(1 - 2\theta)\varphi} \,\middle|\, 
			z = re^{i\varphi}, \varphi \in (-\pi, \pi], r \ge 0, \theta \in [0, 1]
		\right\}
	\]
	and the left-hand arc between $z$ and $\conj{z}$ as
	\[
		\larc(z, \conj{z})	= -\rarc(-z, -\conj{z}).
	\]
        Note that both $r$ and $\varphi$ used in defining the right-hand arc are uniquely determined unless $z = 0$, so the set defines an arc, not a cone.
	If an SRG-full class $\cA$ satisfies 
	\[
		\rarc(z, \conj{z})	\subseteq \srg{\cA}, \qquad \forall z \in \srg{\cA}\setminus 
		\{\infty\},
	\]
	we say $\cA$ satisfies the \emph{right-arc property}. Similarly, if 
	\[
		\larc(z, \conj{z})	\subseteq \srg{\cA}, \qquad \forall z \in \srg{\cA}\setminus 
		\{\infty\},
	\]
	holds, we say $\cA$ has the \emph{left-arc property}.
	If an SRG-full class $\cA$ has either the right-arc property or the left-arc property,
	then it is said to have \emph{an arc property}.

	\paragraph{\textbf{The maximum modulus theorem}}
    Define the boundary of $A \subseteq \complex$ as $\partial A = \mathrm{cl}{A} \setminus \mathrm{int} A$,
    where $\mathrm{cl}{A}$ and $\mathrm{int} A$ are the closure and interior of $A$, respectively.

	\begin{fact}[The maximum modulus theorem \cite{SteinShakarchi2010_complex}]\label{fact:max-mod}
		Let $D\subset\complex$ be a compact set.
		If $h\colon\complex  \rightarrow\complex$ is holomorphic on all of $\,\complex$, then $\max_{z\in D}|h(z)|$ is attained at some point of $\partial D$.
	\end{fact}
	

	
	This maximum modulus theorem simplifies our calculations of Sections~\ref{s:dys_cont} and \ref{s:dys_avg} via the following lemma.
	
	\begin{lemma}\label{lem:max-mod}
	Let $f \colon \complex^3 \to \complex$ be a polynomial of three complex variables.
	Let $A$, $B$, and $C$ be compact subsets of $\,\complex$.
	Then, 
	\[
		\max_{\substack{z_A \in A , z_B \in B ,\\ z_C \in C}} \abs{f(z_A, z_B, z_C)}
		=
		\max_{\substack{z_A \in \partial A , z_B \in \partial B ,\\ z_C \in \partial C}} \abs{f(z_A, z_B, z_C)}.
	\]
	\end{lemma}
\begin{proof}
    As $f$ is a polynomial, $\abs{f}$ is continuous on its compact domain $A \times B \times C \subset \complex^3$. Hence, $\abs{f}$ attains its maximum at some point  $(z_A^0, z_B^0, z_C^0) \in A \times B \times C$.
    
    Furthermore, $f$ is holomorphic in each argument with the other two arguments fixed as it is a polynomial.
	Then, if we fix $z_B = z_B^0$ and $z_C = z_C^0$,
	there exists some $z_A^1 \in \partial A$ such that 
	\[
	    \abs{f(z_A^1, z_B^0, z_C^0)}
	    = \abs{f(z_A^0, z_B^0, z_C^0)}
	\]
	by Fact~\ref{fact:max-mod}. Applying similar procedure with respect to other arguments, we obtain $z_B^1 \in \partial B$ and $z_C^1 \in \partial C$ such that
	\begin{align*}
	    \abs{f(z_A^1, z_B^1, z_C^1)}
	    &= \abs{f(z_A^0, z_B^0, z_C^0)}\\
	    &= \max_{z_A \in A, z_B \in B, z_C \in C}
	    \, \abs{f(z_A, z_B, z_C)}.
	\end{align*}
	
	\end{proof}
 
\section{SRG theory}
	\label{s:srg_theory}
	Let
	\begin{align*}
		\opT_{\opA, \opB, \opC, \alpha, \lambda} 
		&=
		(1 - \lambda)\opI + \lambda\left( 
			\opI - \resa{\opB} + \resa{\opA}(2\resa{\opB} - \opI - \alpha \opC \resa{\opB})
		 \right) \\
		&=
		\opI - \lambda \resa{\opB} + \lambda \resa{\opA} (2\resa{\opB} - \opI - \alpha \opC \resa{\opB}) 
	\end{align*}
	be the (averaged) DYS operator for operators $\opA\colon \cH\rightrightarrows\cH$, $\opB\colon \cH\rightrightarrows\cH$, and $\opC\colon \cH\rightrightarrows\cH$ 
	with stepsize $\alpha \in (0, \infty)$ and averaging parameter $\lambda \in (0, \infty)$.
Let \[
		\opT_{\cA, \cB, \cC, \alpha, \lambda} = \left\{ 
			\opT_{\opA, \opB, \opC, \alpha, \lambda} \, |\, \opA \in \cA, \opB \in \cB, \opC \in \cC
			\right\}
	\]
	 be the class of DYS operators for operator classes $\cA$, $\cB$, and $\cC$ with $\alpha, \lambda \in (0, \infty)$.
	 Analogously, define
	\begin{align*}
		\zeta_{\text{DYS}}(z_A, z_B, z_C; \alpha, \lambda) 
		&= 1 - \lambda z_B + \lambda z_A (2z_B - 1 - \alpha z_C z_B)\\
		&= 1 - \lambda z_A - \lambda z_B + \lambda (2- \alpha z_C)z_A z_B
	\end{align*}
	and
	\[
		\cZ^{\text{DYS}}_{\cA, \cB, \cC, \alpha, \lambda} = 
		\left\{ \zeta_{\text{DYS}}(z_A, z_B, z_C; \alpha, \lambda)
		\,|\, z_A \in \srg{\resa{\cA}}, z_B \in \srg{\resa{\cB}}, z_C \in \srg{\cC}
		\right\}
	\]
   for operator classes $\cA$, $\cB$, and $\cC$ with $\alpha, \lambda \in (0, \infty)$.
   In the Theorems~\ref{thm:lipschitz-dys-a-lip-sm}, \ref{thm:lipschitz-dys-a-lip-b-sm}, and \ref{thm:lipschitz-dys-a-lip-c-sm} that we later prove, the fact that $\zeta_{\text{DYS}}(z_A, z_B, z_C; \alpha, \lambda) $ is symmetric in $z_A$ and $z_B$ will play an important role.


	

In this section, we establish the correspondence between $\opT_{\cA, \cB, \cC, \alpha, \lambda}$ and $\cZ^{\text{DYS}}_{\cA, \cB, \cC, \alpha, \lambda} $ through the following theorem.
	\begin{theorem}\label{thm:tight-dys-main}
Let $\alpha, \lambda \in (0, \infty)$.
		Let $\cA$ and $\cB$ be SRG-full classes of monotone operators and assume $\opI + \alpha\cA$ has the right-arc property.
		Let $\cC$ be an SRG-full class of single-valued operators.
            Assume $\srg{\cA}$, $\srg{\cB}$, $\srg{\cC} \ne \emptyset$.
		If $s \ne 1$ is a real number such that $(2 - \lambda(1-s)^{-1})\opI - \alpha \cC$ has an arc property, then 
		\[
		\abs{\srg{\opT_{\cA, \cB, \cC, \alpha, \lambda}}-s} = \left| \cZ^{\text{DYS}}_{\cA, \cB, \cC, \alpha, \lambda} - s \right|.
		\]
	\end{theorem}

	In Section~\ref{s:dys_avg}, we use Theorem~\ref{thm:tight-dys-main} with $s\ne 0$.
 In Section~\ref{s:dys_cont}, we use Theorem~\ref{thm:tight-dys-main} with $s=0$ to obtain contraction factors of the DYS operator by bounding the maximum modulus of $\cZ^{\text{DYS}}_{\cA, \cB, \cC, \alpha, \lambda}$ as described in the following corollary.
	
	\begin{corollary}\label{cor:tight-coeff-dys}
		Let $\cA$, $\cB$, $\cC$, $\alpha$, and $\lambda$ be as in Theorem~\ref{thm:tight-dys-main} and $(2 - \lambda)\opI - \alpha \cC$ has an arc property.)
		Then,
		\[
			\sup_{\substack{\opT \in \opT_{\cA, \cB, \cC, \alpha, \lambda} \\ x, y \in \dom\opT,\, x \ne y}}
			\frac{\norm{\opT x - \opT y}}{\norm{x - y}}
			= \sup_{z \in \cZ^{\text{DYS}}_{\cA, \cB, \cC, \alpha, \lambda}} \abs{z}.
		\]
	\end{corollary}
	\begin{proof}
		With $s = 0$ in Theorem~\ref{thm:tight-dys-main}, we have
		\[
			\abs{\srg{\opT_{\cA, \cB, \cC, \alpha, \lambda}}} = \left| \cZ^{\text{DYS}}_{\cA, \cB, \cC, \alpha, \lambda} \right|.
		\]
		and the left-hand side can be expressed as
		\[
			\abs{\srg{\opT_{\cA, \cB, \cC, \alpha, \lambda}}} = 
			\left\{ \frac{\norm{\opT x - \opT y}}{\norm{x - y}} 
			\, \middle|\,
			\opT \in \opT_{\cA, \cB, \cC, \alpha, \lambda},\,  x, y \in \dom\opT,\, x \ne y \right\}.
		\]
	
	\end{proof}

Notably, Corollary~\ref{cor:tight-coeff-dys} allows one to compute the tight contraction factor by calculating the maximum modulus of a set of complex numbers $\cZ^{\text{DYS}}_{\cA, \cB, \cC, \alpha, \lambda}$ or, equivalently, by bounding the maximum of $ |\zeta_{\text{DYS}}(z_A, z_B, z_C; \alpha, \lambda)|$ subject to the constraints $z_A \in \srg{\resa{\cA}}, $ $z_B \in \srg{\resa{\cB}}, $ and $z_C \in \srg{\cC}$ on the inputs to the complex polynomial $\zeta_{\text{DYS}}$.

The arc property requirement on $\cC$ holds under the standard assumption that $\opC$ is cocoercive. 
Specifically, if $\alpha$, $\beta_\cC \in (0, \infty)$ and $\cC = \coco{\beta_\cC}$, then $\srg{(2 - \lambda)\opI - \alpha \cC}$ is a disk with its center on the real axis, thus an arc property is satisfied.
However, the conclusion of Corollary~\ref{cor:tight-coeff-dys} may fail if the arc property is not satisfied, as we demonstrate in the arXiv version of this paper \cite{lee2022convergence} by comparing the maximum modulus with tight factor obtained using the performance estimation problem (PEP) framework \cite{drori2014performance,taylor2017smooth,ryu_operator_2020}. 

When the arc property fails, we can consider an enlarged class $\cC' \supset \cC$ such that $(2 - \lambda)\opI - \alpha \cC'$ does satisfy an arc property.
	
	\begin{corollary}\label{cor:tight-coeff-dys-completion}
		Let $\cA$, $\cB$, $\alpha$, and $\lambda$ be as in Theorem~\ref{thm:tight-dys-main} and let $s=0$.
		Assume $\cC$ is a class of single-valued operators, but do not assume that $(2 - \lambda)\opI - \alpha \cC$ has an arc property.
		If there exists some SRG-full class $\cC'$ of single-valued operators such that
		\[
			 \cC'\supseteq \cC, \quad (2 - \lambda)\opI - \alpha\cC'	\text{ has an arc property}
		\]
		then 
		\[
			\sup_{\substack{\opT \in \opT_{\cA, \cB, \cC, \alpha, \lambda} \\ x, y \in \dom\opT, x \ne y}}
			\frac{\norm{\opT x - \opT y}}{\norm{x - y}}
			\le \sup_{z \in \cZ^{\text{DYS}}_{\cA, \cB, \cC', \alpha, \lambda}} \abs{z}.
		\]
	 (The right-hand side depends on $\cC'$, rather than $\cC$.)
	\end{corollary}
	\begin{proof}
		Since $\cC \subseteq \cC'$, we have $\opT_{\cA, \cB, \cC, \alpha, \lambda} \subseteq \opT_{\cA, \cB, \cC', \alpha, \lambda}$.
		Therefore,
		\[
			\sup_{\substack{\opT \in \opT_{\cA, \cB, \cC, \alpha, \lambda} \\ x, y \in \dom\opT,\, x \ne y}}
			\frac{\norm{\opT x - \opT y}}{\norm{x - y}}
			\le
			\sup_{\substack{\opT \in \opT_{\cA, \cB, \cC', \alpha, \lambda} \\ x, y \in \dom\opT,\, x \ne y}}
			\frac{\norm{\opT x - \opT y}}{\norm{x - y}}
			= \sup_{z \in \cZ^{\text{DYS}}_{\cA, \cB, \cC', \alpha, \lambda}} \abs{z},
		\]
		where the last equality follows from Corollary~\ref{cor:tight-coeff-dys}.
	
	\end{proof}
	Note that we can apply Theorem~\ref{thm:tight-dys-main}, as well as Corollaries~\ref{cor:tight-coeff-dys} and \ref{cor:tight-coeff-dys-completion}, for $\cA = \sm{\mu_\cA} \cap \coco{\beta_\cA} \cap \lip{L_\cA}$, where $\mu_\cA \in [0, \infty)$, 
	$\beta_\cA \in [0, \infty)$, and $L_\cA \in (0, \infty]$. (This includes classes like $\cM \cap \lip{L_\cA}$ by setting $\mu_\cA = \beta_\cA = 0$.)
	To see this, observe that SRGs of $\opI + \alpha\sm{\mu_\cA}$, $\opI + \alpha\coco{\beta_\cA}$, and $\opI + \alpha\lip{L_\cA}$ are in form of a disk with its center being a nonnegative real number or $\left\{ z \in \complex\,|\,\Re z \ge a \right\}$ for some nonnegative $a$ and     thus satisfy the right-arc property. 
	Therefore, $\opI + \alpha\cA$, as an intersection of them, is an SRG-full class of monotone operators satisfying the right-arc property.

	The remainder of this section proves Theorem~\ref{thm:tight-dys-main} with the following structure.
	Section~\ref{ss:quotient-operator} defines the quotient operator and characterizes some useful properties, 
	Section~\ref{ss:qo-properties} establishes an SRG product theorem and apply it to the quotient operator,
	and Section~\ref{ss:pf-main-thm} proves Theorem~\ref{thm:tight-dys-main}.

	\subsection{Quotient operator}\label{ss:quotient-operator}
	For $\opA \colon \cH \rightrightarrows \cH$ and $\opB \colon \cH \rightrightarrows \cH$, define the \emph{quotient operator of $\opA$ over  $\opB$} as 
	\[
		\frac{\opA}{\opB} = \opA / \opB = \opA \opB^{-1}.
	\]
	Correspondingly, for a class of operators $\cA$, define
	\[
	    \frac{\cA}{\opB} = \cA/\opB = \cA\opB^{-1}.
	\]
	For a single-valued operator $\opA\colon \cH \rightrightarrows \cH$, define the \emph{selection} of $\opA$ with respect to distinct $x$, $y \in \cH$ as
	\[
		\opA_{x,y} = 
            \begin{cases}
                \{(x, \opA x), (y, \opA y)\} & \text{if } x, y \in \dom\opA, \\
                \{(x, \opA x)\} & \text{if } x \in \dom\opA, y \notin \dom\opA, \\
                \{(y, \opA y)\} & \text{if } x \notin \dom\opA, y \in \dom\opA, \\
                \emptyset & \text{otherwise}.
            \end{cases}
	\]
    We clarify that $\opA_{x,y} = \opA/\opI_{x,y}$.
    
	We point out some immediate consequences.
	Let $\opA \colon \cH \rightrightarrows \cH$ and $\opB \colon \cH \rightrightarrows \cH$ be single-valued operators. 
	Then, for distinct $x$, $y \in \cH$,
	\begin{align}\label{eq:quotient-op-value}
		\srg{\opA / \opB_{x,y}} &=
		\begin{cases}
		    \diffnormfrac{\opA}{\opB}{x}{y}
		    \exp[\pm i \diffang{\opA}{\opB}{x}{y}]
		    &
		    \text{if } \opB x \neq \opB y,
		    \\
		    \infty
		    &
		    \text{if }\opA x \neq \opA y, \opB x = \opB y,
		    \\
		    \emptyset
		    &
		    \text{otherwise.}
		\end{cases}
	\end{align}
    If $\opA \colon \cH \rightrightarrows \cH$ is a single-valued operator, then
    \begin{equation}
            \bigcup_{\substack{x, y \in \cH \\ x \neq y }}\srg{\opA/\opI_{x,y}}=\srg{\opA}.
            \label{eq:srg-is-union-of-sel-ops}
    \end{equation}
	
        Informally, the following lemma says that for SRG-full classes, we can recover the entire SRG from a single $x,y$ pair.
	\begin{lemma}\label{lem:srg-full-class-and-quotient}
	Let $\cH$ be a real Hilbert space with $\dim \cH \ge 2$.
	For any distinct $x, y \in \cH$ and
	an SRG-full class $\cA$ consisting of single-valued operators,
	\[
		\srg{\cA/ \opI_{x,y}} = \srg{\cA}.
	\]
	\end{lemma}
	\begin{proof}
	For any $\opA \colon \cH \rightrightarrows \cH$ where $\opA \in \cA$ , we have $\srg{\opA / \opI_{x, y}} \subseteq \srg{\opA}$ from \eqref{eq:srg-is-union-of-sel-ops}, and $\srg{\cA / \opI_{x, y}} \subseteq \srg{\cA}$.
	
	We now prove $\srg{\cA} \subseteq \srg{\cA / \opI_{x, y}}$. Let $z_A \in \srg{\cA}$ with $\Im z_A \ge 0$. 
	Then, there exist $u, v \in \cH$ such that 
	\[
		z_A= \frac{\norm{u - v}}{\norm{x - y}}\exp[i\angle(u - v, x - y)],
	\]
	since $\dim \cH \ge 2$.
	By Fact~\ref{fact:interpolable}, there exists $\opA \in \cA$ such that $\opA \colon \cH \rightrightarrows \cH$ and $\opA x = u$, $\opA y = v$. Hence, 
	$z_A \in \srg{\opA / \opI_{x, y}} \subseteq \srg{\cA / \opI_{x, y}}$. Since SRGs are symmetric about the real axis, we can apply the same reasoning for $z_A \in \srg{\cA}$ with $\Im z_A < 0$.
	 Hence, $\srg{\cA} \subseteq \srg{\cA / \opI_{x, y}}$, and we conclude the proof.
	\end{proof}
	
	We now introduce some rules of calculus with the quotient operators.
	
	\begin{lemma}{(Calculus of quotient operators)}\label{lem:quotient-calculus}
	Let $\opA$, $\opB$, $\opC$, $\opD \colon \cH \rightrightarrows \cH$ be single-valued operators and distinct $x$, $y \in \cH$.
	Then the following identities hold:
	\begin{align}
		\srg{\opA / \opB_{x,y}} &= \srg{\opB / \opA_{x,y}}^{-1}  \label{prop1-1}\\
		\para{\opA + t\opB} / \opB_{x,y} &= \opA / \opB_{x,y} + t\opI   &(t \in \R) \label{prop1-2}\\
		\para{t\opA} / \opB_{x, y} &= t \para{\opA /\opB_{x, y}}   &(t \in \R)\label{prop1-3}\\
		\abs{\srg{\opA / \opC_{x,y}}} &=\abs{\srg{\opA / \opB_{x, y}}}\abs{\srg{\opB / \opC_{x,y}}} &(\opB x \neq \opB y, \opC x \neq \opC y)\label{prop1-4}\\
		\srg{(\opA + \opD) / \opC_{x, y}} &= \srg{\opA / \opC_{x, y}} &(\opD x = \opD y)\label{prop1-5}.
	\end{align}
	\end{lemma}
	
	\begin{proof}
	In our proofs of \eqref{prop1-1}, \eqref{prop1-2}, and \eqref{prop1-3}, we assume $x ,y\in \dom\opA \cap \dom\opB$, as otherwise the statements are straightforwardly true due to \eqref{eq:quotient-op-value}. We assume similar conditions for \eqref{prop1-5} regarding domains of $\opA$ and $\opC$.
	
        \eqref{prop1-1} If $ \opA x \neq \opA y$ and $ \opB x \neq \opB y$, 
	\begin{align*}
			\srg{\opA / \opB_{x,y}} &= \frac{\norm{\opA x - \opA y}}{\norm{\opB x - \opB y}}\exp[\pm i \angle(\opA x - \opA y, \opB x - \opB y)]\\
			&= \pths{ \frac{\norm{\opB x - \opB y}}{\norm{\opA x - \opA y}} \exp[\pm i\angle(\opA x - \opA y, \opB x - \opB y)] }^{-1}\\ 
			&= \srg{\opB / \opA_{x,y}}^{-1}.
	\end{align*}
	If $ \opA x = \opA y$ and $ \opB x \neq \opB y$, then $\srg{\opA / \opB_{x,y}}=0=\infty^{-1}=\srg{\opB / \opA_{x,y}}^{-1}$.\\
	If $ \opA x \ne  \opA y$ and $ \opB x = \opB y$, then $\srg{\opA / \opB_{x,y}}=\infty=0^{-1}=\srg{\opB / \opA_{x,y}}^{-1}$.

	\eqref{prop1-2} Both sides evaluate to 
	$\opA x + t\opB x$ and $\opA y + t\opB y$ for inputs $\opB x$ and $\opB y$, respectively. Otherwise, both sides evaluate to $\emptyset$.
	
	\eqref{prop1-3} This follows immediately from our operator notation.
	
	\eqref{prop1-4} This follows from 
	\begin{align*}
		\abs{\srg{\opA / \opC_{x,y}}} 
		&= \frac{\norm{\opA x - \opA y}}{\norm{\opC x - \opC y}}\\
		&= \frac{\norm{\opA x - \opA y}}{\norm{\opB x - \opB y}}
		\frac{\norm{\opB x - \opB y}}{\norm{\opC x - \opC y}}\\
		&= \abs{\srg{\opA / \opB_{x,y}}} \abs{\srg{\opB / \opC_{x,y}}}.
	\end{align*}
	
	\eqref{prop1-5} Observe that 
	\[
	    (\opA + \opD)x - (\opA + \opD)y = \opA x - \opA y,
	\]
	which, together with \eqref{eq:quotient-op-value}, gives the desired result.
	\end{proof}

	\subsection{SRG product theorem}
	\label{ss:qo-properties}
	
	We first establish a product theorem for SRGs, which is of independent interest. This product theorem differs from  \cite[Theorem~7]{ryu_scaled_2021} in that one of operands  is a single fixed operator, rather than an operator class.
	We then apply it to the quotient operator.

	\begin{theorem}[SRG product theorem]\label{thm:split-srg}
	Let $\cH$ be a real Hilbert space with $\textrm{dim}\,\cH \ge 2$ and $\opB \colon \cH \rightrightarrows \cH$ be a single-valued operator.
        Let $\cA$ be a SRG-full class of single-valued operators satisfying an arc property, and assume $\srg{\cA} \ne \emptyset$.
        Then, 
	\[
		\srg{\cA\opB} = \srg{\cA}\srg{\opB}.
	\]
	(Note, $\infty\notin \cG(\cA)$ since operators in $\cA$ are single-valued.)
	\end{theorem}

	\begin{proof}
	We first show $\srg{\cA}\srg{\opB} \subseteq \srg{\cA\opB}$.
	Let $z_A\in \srg{\cA}$ and $z_B\in \srg{\opB}$.
	If $z_B=0$, there exists distinct $x, y\in \cH$ such that $\opB x=\opB y$.
        Take $\opA = \{(\opB x, 0)\}$. Then $\opA \in \cA$ by Fact~\ref{fact:single-pt-op}.
        This gives
        \[
            z_A z_B = 0 = \frac{\norm{\opA\opB x - \opA\opB y}}{\norm{x - y}}\exp[i\angle(\opA\opB x - \opA\opB y, x - y)] \in \srg{\opA\opB}.
        \]
        
	Otherwise, assume $z_B\ne 0$.
	Then there exists distinct $x,y\in \cH$ such that  $\opB x\ne \opB y$ and
	\[
	z_B \in \frac{\|\opB x-\opB y\|}{\|x-y\|}\exp\left[\pm i \angle (\opB x-\opB y,x-y)\right].
	\]
	There also exists orthonormal vectors $\{e_1,e_2\}\subset\cH$ such that 
    \begin{align*}
    e_1 &= (x - y) / \norm{x - y}\\
    \opB x - \opB y &= \norm{x - y}((\Re z_B) e_1 + (\Im z_B) e_2).
    \end{align*}
	With the $z_A\in \cG(\cA)$, define 
	\begin{align*}
		u &= \norm{x - y} ((\Re (z_A z_B)) e_1 + (\Im (z_A z_B)) e_2) \\
		v &= 0.
	\end{align*}
        
	The definitions of $u$ and $v$ imply
        \[
		z_A z_B \in \frac{\norm{u - v}}{\norm{x - y}}\exp[\pm i \angle(u - v, x - y)]
	\]
        and
        \[
            \frac{\norm{u - v}}{\norm{\opB x - \opB y}} \exp[i \angle (u - v, \opB x - \opB y)] \in \{z_A, \conj{z_A}\} \subseteq \srg{\cA}.
        \]
	By Fact~\ref{fact:interpolable}, there exist $\opA \in \cA$ such that $\opA \colon \cH \rightrightarrows \cH$ and 
	\[
	   \opA \opB x = u, \quad \opA \opB y = v.
	\]
        \begin{center}
		\def\scale{2}
		\begin{tikzpicture}[scale=\scale]
			\def\a{1.6};
			\def\b{0.8};
			\def\c{1};
			\def\t{30};
			\def\w{80};

			\coordinate (O) at (0, 0);
			\coordinate (C) at ({\c}, {0});
			\coordinate (B) at ({\b*cos(\t)}, {\b*cos(\t)});
			\coordinate (U1) at ({\a*\b*cos(\t+\w)}, {\a*\b*sin(\t+\w)});
			
			\filldraw (O) circle[radius={\ptsize / \scale pt}];

			\draw (O) node[below left=0pt] {$0$};
			\draw (B) node[below right=0pt] {$\opB x - \opB y$};
			\draw (C) node[below=0pt] {$x - y$};
			\draw (U1) node[left=0pt] {$u - v = \opA\opB x - \opA\opB y$};

			\draw [->] (O) -- (B);
			\draw [->] (O) -- (C);
			\draw [->] (O) -- (U1);

			\pic [draw, angle radius=\scale*0.4*\cminpt] {angle=B--O--U1};
			\pic [draw, angle radius=\scale*0.35*\cminpt] {angle=C--O--B};
			
                \pic [draw, angle radius=\scale*0.3*\cminpt] {angle=C--O--B};

                \draw ({0.45*cos(20)},{0.45*sin(20)}) node[right=-5pt] {$\arg z_B$};
			\draw [dashed,<-] (-1.5,0) -- (0,0);
			\draw [dashed,->] (C) -- (1.5,0);
			\draw (1.5,0) node[above=0pt] {$e_1$};
			\draw [dashed,<->] (0,-0.3) -- (0,1.5);
			\draw (0,1.5) node[above=0pt] {$e_2$};
			\draw [fill, white] ({0.6*cos(100)},{0.45*sin(100)}) rectangle ({0.3},{0.65*sin(100)});
			\draw ({0.5*cos(100)},{0.53*sin(100)}) node[right=-5pt] {$\arg z_A$};
		\end{tikzpicture}	
	\end{center}
	We conclude
	\begin{align*}
			z_A z_B 
			&\in \frac{\norm{\opA\opB x- \opA\opB y}}{\norm{x- y}}\exp[\pm i\angle(\opA\opB x- \opA\opB y, x- y))]\\
			&\subseteq  \cG(\opA\opB).
	\end{align*}
	

	Next, we show $\srg{\cA\opB} \subseteq \srg{\cA}\srg{\opB}$.
	Consider any $\opA \colon \cH \rightrightarrows \cH$ satisfying $\opA \in \cA$. We will show $\srg{\opA\opB} \subseteq \srg{\cA}\srg{\opB}$, 
	and, by \eqref{eq:srg-is-union-of-sel-ops}, this is equivalent to
	\[
	\srg{\opA\opB/\opI_{x, y}} \subseteq \srg{\cA}\srg{\opB}
	\]
	for all distinct $x,y \in \dom\opA\opB$.

	Let $x,y \in \dom\opA\opB$ be distinct.
    If $\opB x = \opB y$, then $0 \in \srg{\opB}$ and $\opA\opB x = \opA\opB y$. Thus,
        \[
            \srg{\opA\opB / \opI_{x, y}} = \{0\} \subseteq \srg{\cA}\srg{\opB}.
        \]
    Next, consider the case $\opB x \ne \opB y$.
    Define 
        \[
            z_{AB} = \srgval{\opA\opB}{\opI}{x}{y}
        \]
        so that $\srg{\opA\opB / \opI_{x, y}} = \{ z_{AB}, \conj{z_{AB}} \}$.
        Furthermore, we write
	\begin{align*}
		\varphi_B &= \angle(\opB x - \opB y, x - y), \\
		\varphi_{AB/B} &= \angle(\opA\opB x- \opA\opB y, \opB x- \opB y), \\
		z_B &= \frac{\norm{\opB x - \opB y}}{\norm{x - y}} \exp[i \varphi_B], \\
		z_{AB/B} &= \frac{\norm{\opA\opB x - \opA\opB y}}{\norm{\opB x - \opB y}} \exp[i \varphi_{AB/B}].
	\end{align*}
        Observe that $z_B$, $\conj{z_B} \in \srg{\opB}$ and $z_{AB/B} \in \srg{\opA} \subseteq \srg{\cA}$.

	First, suppose $\cA$ has the right-arc property.
	By Fact~\ref{fact:sphere-tri-ineq}, 
	\[
		\angle\left( \opA\opB x - \opA\opB y, x - y \right)
		\in \left[ \abs{\varphi_B - \varphi_{AB/B}}, \varphi_B + \varphi_{AB/B} \right].
	\]
	Therefore,
	\begin{align*}
		z_{AB} 
		&\in \frac{\norm{\opA\opB x-\opA\opB y}}{\norm{x- y}} \exp\left[ {i \left[ |\varphi_B-\varphi_{AB/B}|, \varphi_B+\varphi_{AB/B} \right]} \right]\\
		&= \frac{\norm{\opB x-\opB y}}{\norm{x - y}}\frac{\norm{\opA\opB x-\opA\opB y}}{\norm{\opB x- \opB y}} \exp\left[ {i \left[ |\varphi_B-\varphi_{AB/B}|, \varphi_B+\varphi_{AB/B} \right]} \right]\\
		& \subseteq \frac{\norm{\opB x-\opB y}}{\norm{x - y}}\frac{\norm{\opA\opB x-\opA\opB y}}{\norm{\opB x- \opB y}} \exp\left[ i \left[ \varphi_B-\varphi_{AB/B}, \varphi_B+\varphi_{AB/B} \right] \right]\\
		& = z_B \rarc(z_{AB/B}, \conj{z_{AB/B}})\\ 
		&\subseteq \srg{\opB}\srg{\cA}
	\end{align*}
	by the right-arc property of $\cA$.
	Furthermore, 
	\begin{align*}
	    \conj{z_{AB}} &\in \conj{z_B} \conj{\rarc(z_{AB/B}, \conj{z_{AB/B}})}\\
		&=\conj{z_B} \rarc(z_{AB/B}, \conj{z_{AB/B}}) \\
		&\subseteq \srg{\opB}\srg{\cA}.
	\end{align*} 
	Thus, $\srg{\opA\opB/\opI_{x, y}} = \{z_{AB}, \conj{z_{AB}}\} \subseteq \srg{\cA}\srg{\opB}$. 
	
	If  $\cA$ has the left-arc property, $-\cA$ has the right-arc property.
	As $-\opA \in -\cA$,
	\[
	    \srg{\opA\opB/\opI_{x, y}} = -\srg{-\opA\opB/\opI_{x, y}} \subseteq -\srg{-\cA}\srg{\opB} = \srg{\cA}\srg{\opB}
	\]
	by Fact~\ref{fact:srg-trans}.
	Therefore, by \eqref{eq:srg-is-union-of-sel-ops}, we conclude $\srg{\cA\opB} \subseteq \srg{\cA}\srg{\opB}$ when $\cA$ has an arc property.	
	\end{proof}

    
    \begin{corollary}\label{cor:split-quot-srg}
    Let $\cA$ and $\cH$ be as in Theorem~\ref{thm:split-srg}. 
    Let $\opB \colon \cH \rightrightarrows \cH$ and $\opC \colon \cH \rightrightarrows \cH$ be single-valued operators.
	For $\opC x \ne \opC y$,
	\[
	    \bigcup_{\substack{\opA \colon \cH \rightrightarrows \cH \\ \opA \in \cA}} \srg{\frac{\opA\opB}{\opC_{x, y}}} = \srg{\frac{\cA\opB}{\opC_{x, y}}} = \srg{\cA}\srg{\frac{\opB}{\opC_{x, y}}}
	\]
	holds.     
    \end{corollary}	
    
    \begin{proof}
        The first identity holds by definition, and the second identity follows from Theorem~\ref{thm:split-srg} with replacing
        $\opB$ in the theorem statement to $\opB / \opC_{x, y}$.
    \end{proof}
	

	
	\subsection{Proof of Theorem~\ref{thm:tight-dys-main}}
	\label{ss:pf-main-thm}

	We have one final preliminary lemma to establish.
	The proof of Theorem~\ref{thm:tight-dys-main} primarily focuses on operators on $\cH$ such that $\dim\cH\ge 2$ (including the case $\dim\cH=\infty$). The following lemma allows to contain the case $\dim \cH=1$.
	
	\begin{lemma}\label{lem:ignore-one-dim}
	Let $\cA, \cB, \cC$ be SRG-full classes and $\alpha,\lambda \in (0, \infty)$.
	Let $\cH$ be a real Hilbert space. Then,
	\[
		\bigcup_{\substack{
		\opA,\opB,\opC \colon \cH \rightrightarrows \cH\\
		\opA \in \cA,\, \opB \in \cB,\, \opC \in \cC}}
		\srg{\opT_{\opA, \opB, \opC, \alpha, \lambda}}
		\subseteq
		\bigcup_{\substack{
		\opA',\opB',\opC' \colon \cH \times \cH \rightrightarrows \cH \times \cH \\
		\opA' \in \cA,\, \opB' \in \cB,\, \opC' \in \cC}}
		\srg{\opT_{\opA', \opB', \opC', \alpha, \lambda}}.
	\]
	To clarify, $\cH \times \cH$ is a real Hilbert space with $\dim(\cH \times \cH) = 2\dim\cH$.
	\end{lemma}

	\begin{proof}
		For an operator $\opA \colon \cH \rightrightarrows \cH$, define the embedding $\imath$ as
		\begin{align*}
		    \imath(\opA)&\colon \cH\times \cH \rightrightarrows \cH\times \cH\\
			\imath(\opA)&= \left\{ 
				\left((x, 0), (u, 0)\right) \,|\, (x, u) \in \opA
			\right\},
		\end{align*}
		where $0\in \cH$.
		Intuitively speaking, $\imath(\opA)$ is essentially the same operator as $\opA$, with $\mathrm{graph}(\opA)\subseteq\cH\times\cH$ embedded into $(\cH\times\{0\}) \times (\cH\times\{0\})$.
		One can observe that
	\begin{gather*}
			\imath(\opA + \opB) = \imath(\opA) + \imath(\opB),\quad
			\imath(\opA\opB) = \imath(\opA)\imath(\opB),
			\\
			\imath(\alpha\opA) = \alpha\imath(\opA),\quad
			\imath(\opI + \opA) = \opI + \imath(\opA),\quad
			\imath(\opA^{-1}) = \imath(\opA)^{-1},\\
			\imath(\resa{\opA}) = \resa{\imath(\opA)}, \quad \imath(\opT_{\opA, \opB, \opC, \alpha, \lambda})
			= \opT_{\imath(\opA), \imath(\opB), \imath(\opC), \alpha, \lambda}.
		\end{gather*}
            We omit the proof of the first five statements for conciseness, which can be done by a simple algebra.
            The last two expressions are consequences of the former statements. 
            In particular, observe that
            \[
                \imath(\resa{\opA}) = \imath((\opI + \alpha \opA)^{-1}) = \imath(\opI + \alpha \opA)^{-1}
                = (\opI + \alpha \imath(\opA))^{-1} = \resa{\imath(\opA)},
            \]
            and 
            \begin{align*}
                \imath(\opT_{\opA, \opB, \opC, \alpha, \lambda})
                &= \imath(\opI - \lambda \resa{\opB} + \lambda \resa{\opA} (2 \resa{\opB} - \opI - \alpha \opC \resa{\opB})) \\
                &= \opI - \lambda \imath( \resa{\opB} ) + \lambda \imath (\resa{\opA}) (2 \imath(\resa{\opB}) - \opI - \alpha \imath (\opC) \imath(\resa{\opB})) \\ 
                &= \opI - \lambda \resa{\imath (\opB)} + \lambda \resa{\imath (\opA)} (2 \resa{\imath (\opB)} - \opI - \alpha \imath (\opC) \resa{\imath (\opB)}) \\ 
                &= \opT_{\imath(\opA), \imath(\opB), \imath(\opC), \alpha, \lambda}.
            \end{align*}

        Furthermore, if $\opA \in \cA$, then
		\[
			\srg{\imath(\opA)} = \srg{\opA}\subseteq \srg{\cA}.
		\]
	    Therefore, if $\cA$ is an SRG-full class, then 
	    \[
	    \opA\in \cA\quad\Rightarrow\quad\imath(\opA) \in \cA.
	    \]
		The same argument holds for $\cB$ and $\cC$ as well. Therefore, 
		if $\opA \in \cA$, $\opB \in \cB$, and $\opC \in \cC$,
		\begin{align*}
			\srg{\opT_{\opA, \opB, \opC, \alpha, \lambda}}
			= \srg{\imath(\opT_{\opA, \opB, \opC, \alpha, \lambda})}
			= \srg{\opT_{\imath(\opA), \imath(\opB), \imath(\opC), \alpha, \lambda}}
		\end{align*}
		and $\imath(\opA) \in \cA$, $\imath(\opB) \in \cB$, and $\imath(\opC) \in \cC$,
		which concludes the proof.
	\end{proof}

\begin{proof}[Proof to Theorem~\ref{thm:tight-dys-main}]
		Recall that our goal is to show 
		\[
			\abs{\srg{\opT_{\cA, \cB, \cC, \alpha, \lambda}} - s}
			= \left| \cZ^{\text{DYS}}_{\cA, \cB, \cC, \alpha, \lambda} - s \right|
		\]
		given that $\cA$ and $\cB$ are SRG-full classes of monotone operators,
		$\opI + \alpha\cA$ has the right-arc property,
		and $\cC$ is an SRG-full class of 
		single-valued operators such that $(2 - \lambda(1-s)^{-1})\opI - \alpha \cC$ has an arc property.
		
		Then,
		\begin{align*}
		        \abs{\srg{\opT_{\cA, \cB, \cC, \alpha, \lambda}} - s}
			&\stackrel{\text{(a)}}{=} 
		    \bigcup_{\cH : \substack{\mathrm{Hilbert} \\ \mathrm{space}}}
		    \bigcup_{\substack{\opA, \opB, \opC \colon \cH \rightrightarrows \cH \\ \opA \in \cA,\, \opB \in \cB,\, \opC \in \cC}}
			\left| \srg{\opT_{\opA, \opB, \opC, \alpha, \lambda}} - s \right|\\
			&\stackrel{\text{(b)}}{=} 
			\bigcup_{\cH : \substack{\mathrm{Hilbert} \\ \mathrm{space}\\\dim \cH\ge 2}}
			\bigcup_{\substack{\opA, \opB, \opC \colon \cH \rightrightarrows \cH \\ \opA \in \cA, \, \opB \in \cB, \,\opC \in \cC}}
			\left| \srg{\opT_{\opA, \opB, \opC, \alpha, \lambda}} - s \right|\\
			&\stackrel{\text{(c)}}{=} 
			\bigcup_{\cH : \substack{\mathrm{Hilbert} \\ \mathrm{space}\\\dim \cH\ge 2}}
			\bigcup_{\substack{x, y \in \cH \\ x \ne y}}
			\bigcup_{\substack{\opA, \opB, \opC \colon \cH \rightrightarrows \cH \\ \opA \in \cA,\, \opB \in \cB,\, \opC \in \cC}}
			\abs{\srg{\opT_{\opA, \opB, \opC, \alpha, \lambda}/\opI_{x,y}} - s}\\
			&\stackrel{\text{(d)}}{=} 
			\bigcup_{\cH : \substack{\mathrm{Hilbert} \\ \mathrm{space}\\\dim \cH\ge 2}}
			\bigcup_{\substack{x, y \in \cH \\ x \ne y}}
			\left| \cZ^{\text{DYS}}_{\cA, \cB, \cC, \alpha, \lambda} - s \right|\\
			&\stackrel{\text{(e)}}{=} 
			\left| \cZ^{\text{DYS}}_{\cA, \cB, \cC, \alpha, \lambda} - s \right|,
		\end{align*}
		where (a) follows from definition of $\opT_{\cA, \cB, \cC, \alpha, \lambda}$,
		(b) follows from Lemma~\ref{lem:ignore-one-dim},
		(c) follows from \eqref{eq:srg-is-union-of-sel-ops}, 
		(d) will be established soon,
		and (e) follows from the fact that $\abs{\cZ^{\text{DYS}}_{\cA, \cB, \cC, \alpha, \lambda} - s}$ has no dependency on the bound union variables $\cH$, $x$, and $y$. We establish (d) by showing 
		\[
			\bigcup_{\substack{\opA, \opB, \opC \colon \cH \rightrightarrows \cH \\ \opA \in \cA,\, \opB \in \cB,\, \opC \in \cC}}
			\abs{\srg{\opT_{\opA, \opB, \opC, \alpha, \lambda}/\opI_{x,y}} - s}
				= \left| \cZ^{\text{DYS}}_{\cA, \cB, \cC, \alpha, \lambda} - s \right|	
		\]
		for any	Hilbert space $\cH$ such that $\dim \cH\ge 2$ and distinct $x,y \in \cH$.

		Now, let $s=1-\lambda t$.  Note that $s\ne 1$ implies $t\ne 0$ and we have assumed that $(2 - t^{-1})\opI - \alpha \cC$ has an arc property.
        The assumption that $\opI + \alpha\cA$ has the right-arc property implies that $\resa{\cA}$ has the right-arc property:        If $\opI + \alpha\cA$ satisfies the right-arc property and $z^{-1}\in \cG(\resa{\cA})$, then $\rarc(z, \conj{z})\subseteq \cG(\opI+\alpha\cA)$ and         $(\rarc(z, \conj{z}))^{-1} = \rarc\big( z^{-1}, \conj{z^{-1}} \big)\subseteq \cG(\resa{\cA}) = \cG((\opI + \alpha\cA)^{-1})$. So  $\resa{\cA}$ has the right-arc property.
	
	
	   For (d), we have
	    \begin{align*}
	    &\bigcup_{\substack{\opB \colon \cH \rightrightarrows \cH \\ \opB \in \cB}}\bigcup_{\substack{\opA, \opC \colon \cH \rightrightarrows \cH \\ \opA \in \cA, \,\opC \in \cC}}
			\left|\srg{\frac{\opT_{\opA, \opB, \opC, \alpha, \lambda}}{\opI_{x, y}}} - (1 - \lambda t)\right|
			\nonumber
			\\
			&\stackrel{\text{(i)}}{=} 
			\bigcup_{\substack{\opB \colon \cH \rightrightarrows \cH \\ \opB \in \cB}}
			\bigcup_{\substack{\opA, \opC \colon \cH \rightrightarrows \cH \\ \opA \in \cA, \,\opC \in \cC}}
			\lambda\babs{\srg{\frac{t\opI - \resa{\opB} + \resa{\opA} 
			(2\resa{\opB} - \opI - \alpha \opC \resa{\opB})}{\opI_{x,y}}}}
			\\
			&\stackrel{\text{(ii)}}{=}
			\bigcup_{\substack{\opB \colon \cH \rightrightarrows \cH \\ \opB \in \cB}}
			\bigcup_{z_B \in \srg{{\resa{\opB}}/{\opI_{x, y}}}}
			\lambda
			\left|
			t - z_B + \srg{\resa{\cA}}(2z_B - 1 - \alpha \srg{\cC} z_B)
			\right| 
			\\
			&\stackrel{\text{(iii)}}{=}
			\bigcup_{z_B \in \srg{\resa{\cB}}}
    			\lambda
    			\left|
    			t - z_B + \srg{\resa{\cA}}(2z_B - 1 - \alpha \srg{\cC} z_B)
    			\right|
			\\
			&\stackrel{\text{(iv)}}{=}  \abs{\cZ^{\text{DYS}}_{\cA, \cB, \cC, \alpha, \lambda} - (1 - \lambda t)}.
	    \end{align*}
	    Step (i) follows from
		\begin{align*}
			\left| \srg{\frac{\opT_{\opA, \opB, \opC, \alpha, \lambda}}{\opI_{x,y}}} - (1 - \lambda t) \right|\!\!\!\!\!\!\!\!\!\!\!\!\!\!\!\!\!\!\!\!\!\!\!\!&
			\\
			&= \babs{\srg{\frac{\lambda t\opI - \lambda \resa{\opB} + \lambda \resa{\opA} 
			(2\resa{\opB} - \opI - \alpha \opC \resa{\opB})}{\opI_{x,y}}}}
			\\
			&= \lambda\babs{\srg{\frac{t\opI - \resa{\opB} + \resa{\opA} 
			(2\resa{\opB} - \opI - \alpha \opC \resa{\opB})}{\opI_{x,y}}}}
		\end{align*}
		where we have used \eqref{prop1-2} and \eqref{prop1-3} of Lemma~\ref{lem:quotient-calculus} and Fact~\ref{fact:srg-trans}.
		We discuss (ii) soon.
		Step (iii) follows from
		\[
			\srg{\resa{\cB}} =
			\srg{\resa{\cB} / \opI_{x, y}} =
			\bigcup_{\substack{\opB \colon \cH \rightrightarrows \cH \\ \opB \in \cB}}\srg{{\resa{\opB}}/{\opI_{x, y}}},
		\]
		where we have used Lemma~\ref{lem:srg-full-class-and-quotient}.
		Step (iv) follows from the definition of $\cZ^{\text{DYS}}_{\cA, \cB, \cC, \alpha, \lambda}$.
		
		We now prove (ii) to conclude the proof. We do so by showing
		\begin{align*}
			&\!\!\!\!\!\!\!\!\!\!\!\!\!
			\bigcup_{\substack{\opA, \opC \colon \cH \rightrightarrows \cH \\ \opA \in \cA, \,\opC \in \cC}}
			\babs{\srg{\frac{t\opI - \resa{\opB} + \resa{\opA} 
			(2\resa{\opB} - \opI - \alpha \opC \resa{\opB})}{\opI_{x,y}}}}
			\\
			&
			= 
			\bigcup_{z_B \in \srg{{\resa{\opB}}/{\opI_{x, y}}}}
			\left|
			t - z_B + \srg{\resa{\cA}}(2z_B - 1 - \alpha \srg{\cC} z_B)
			\right|.
		\end{align*}
            If $x \notin \dom\resa{\opB}$ or $y \notin \dom\resa{\opB}$, the both sides equal $\emptyset$.
            We now consider two nontrivial cases.
		First, consider the case $(\resa{\opB}-t\opI)x \neq (\resa{\opB}-t\opI)y$.
		We have
		\begin{align*}
			&\bigcup_{\substack{\opA, \opC \colon \cH \rightrightarrows \cH \\ \opA \in \cA, \,\opC \in \cC}}
			\babs{\srg{\frac{t\opI - \resa{\opB} + \resa{\opA} 
			(2\resa{\opB} - \opI - \alpha \opC \resa{\opB})}{\opI_{x,y}}}} \\
			&\stackrel{(\diamondsuit)}{=}
			\bigcup_{\substack{\opC \colon \cH \rightrightarrows \cH \\ \opC \in \cC}}
			\bigcup_{\substack{\opA \colon \cH \rightrightarrows \cH \\ \opA \in \cA}}
			\babs{\srg{\frac{t\opI - \resa{\opB} +\resa{\opA} 
			(2\resa{\opB} - \opI - \alpha \opC \resa{\opB})}{(\resa{\opB} - t\opI)_{x, y}}}}
			\babs{\srg{\frac{\resa{\opB} - t\opI}{\opI_{x,y}}}}
			\\
			&=\bigcup_{\substack{\opC \colon \cH \rightrightarrows \cH \\ \opC \in \cC}}
			\bigcup_{\substack{\opA \colon \cH \rightrightarrows \cH \\ \opA \in \cA}}
			\babs{\srg{\frac{\resa{\opA} 
			(2\resa{\opB} - \opI - \alpha \opC \resa{\opB})}{(\resa{\opB} - t\opI)_{x, y}}} - 1}
			\babs{\srg{\frac{\resa{\opB} - t\opI}{\opI_{x,y}}}}
			\\
			&\stackrel{(\clubsuit)}{=}
			\bigcup_{\substack{\opC \colon \cH \rightrightarrows \cH \\ \opC \in \cC}}
			\babs{\srg{\resa{\cA}}\srg{\frac{2\resa{\opB} - \opI - \alpha \opC \resa{\opB}}{(\resa{\opB} - t\opI)_{x, y}}} - 1}
			\babs{\srg{\frac{\resa{\opB} - t\opI}{\opI_{x,y}}}}
			\\
			&\stackrel{(\diamondsuit)}{=}
			\bigcup_{\substack{\opC \colon \cH \rightrightarrows \cH \\ \opC \in \cC}}
			\babs{\srg{\resa{\cA}}
			\left[ \srg{
				\frac{
					\left( (2 - t^{-1})\opI - \alpha \opC \right) \resa{\opB}}{(\resa{\opB} - t\opI)_{x, y}}}  + t^{-1} \right]
			- 1}
			\babs{\srg{\frac{\resa{\opB} - t\opI}{\opI_{x,y}}}}
			\\
			&\stackrel{(\clubsuit)}{=}
			\left|\srg{\resa{\cA}}
			\left[
				\left(2-t^{-1} - \alpha\srg{\cC}\right)
				\srg{\frac{\resa{\opB}}{{(\resa{\opB} - t\opI)}_{x,y}}} + t^{-1}
			\right]
			-1 \right|
			\left| \srg{\frac{\resa{\opB}- t\opI}{\opI_{x,y}}} \right| \\
			&\stackrel{(\diamondsuit)}{=}
			\left|\srg{\resa{\cA}}
			\left[
				\left(2-t^{-1} - \alpha\srg{\cC}\right)
			\left(1 + t{\srg{\frac{\opI}{{(\resa{\opB} - t\opI)}_{x,y}}}}\right) + t^{-1}
			\right]
			-1 \right|
			\\&\quad \cdot \left| \srg{\frac{\resa{\opB}- t\opI}{\opI_{x,y}}} \right|
			\\
			&=
			\bigcup_{z_B \in \srg{{\resa{\opB}}/{\opI_{x, y}}}} 
			\left|
			\srg{\resa{\cA}}
			\left[
				\left(2-t^{-1} - \alpha\srg{\cC}\right)
			(1+t(z_B-t)^{-1})+t^{-1}
			\right] -1
			\right|
			\left|z_B - t\right| \\
			&= \bigcup_{z_B \in \srg{{\resa{\opB}}/{\opI_{x, y}}}} 
			\left|
			t - z_B + \srg{\resa{\cA}}(2z_B - 1 - \alpha \srg{\cC} z_B)
			\right|.
		\end{align*}
		where we used \eqref{prop1-2}, \eqref{prop1-3}, and \eqref{prop1-4} of Lemma~\ref{lem:quotient-calculus} in $(\diamondsuit)$'s and
		we used Corollary~\ref{cor:split-quot-srg} and Fact~\ref{fact:srg-trans} in $(\clubsuit)$'s together with the fact that $\resa{\cA}$ and $(2 - t^{-1})\opI - \alpha \cC$ has an arc property.
		The second step is induced by Fact~\ref{fact:srg-trans}. For the second to last step, we utilized \eqref{prop1-1}, \eqref{prop1-2}, Fact~\ref{fact:srg-trans}, and the fact that $\srg{\resa{\opB}/\opI_{x, y}}$ is a set of either one real number or one complex conjugate pair. The last step follows from direct calculations.
		
		Now, consider the case $(\resa{\opB} - t\opI)x= (\resa{\opB} - t\opI)y$.
		In this case, we have
  
		\begingroup
		\allowdisplaybreaks
		\begin{align*}
			&
			\bigcup_{\substack{\opA, \opC \colon \cH \rightrightarrows \cH \\ \opA \in \cA, \,\opC \in \cC}}
			\left|
			\srg{\frac{
			t\opI - \resa{\opB} +
			\resa{\opA}(2\resa{\opB} - \opI - \alpha \opC \resa{\opB})}{\opI_{x,y}}}
			\right| 
			\\
			&\stackrel{(\diamondsuit)}{=}
			\bigcup_{\substack{\opC \colon \cH \rightrightarrows \cH \\ \opC \in \cC}}
			\bigcup_{\substack{\opA \colon \cH \rightrightarrows \cH \\ \opA \in \cA}}
			\left|
			\srg{\frac{\resa{\opA}(2\resa{\opB} - \opI - \alpha \opC \resa{\opB})}{\opI_{x,y}}}
			\right|
			\\
			&\stackrel{(\clubsuit)}{=}
			\bigcup_{\substack{\opC \colon \cH \rightrightarrows \cH \\ \opC \in \cC}}
			\left|
			\srg{\resa{\cA}}
			\srg{\frac{2\resa{\opB} - \opI - \alpha \opC \resa{\opB}}{\opI_{x,y}}}
			\right|
			\\
			&\stackrel{(\diamondsuit)}{=}
			\bigcup_{\substack{\opC \colon \cH \rightrightarrows \cH \\ \opC \in \cC}}
			\left|
			\srg{\resa{\cA}}
			\srg{\frac{((2 - t^{-1})\opI - \alpha \opC)\resa{\opB}}{\opI_{x,y}}}
			\right|
			\\
			&\stackrel{(\clubsuit)}{=}
			\left|
			\srg{\resa{\cA}}((2 - t^{-1})\opI - \alpha \srg{\cC})\srg{\resa{\opB}/\opI_{x,y}}
			\right|
			\\
			&=
			\bigcup_{z_B \in \srg{{\resa{\opB}}/{\opI_{x, y}}}} 
			\left|
			t - z_B + \srg{\resa{\cA}}(2z_B - 1 - \alpha \srg{\cC} z_B)
			\right|,
		\end{align*}
		\endgroup
  where we used \eqref{prop1-5} of Lemma~\ref{lem:quotient-calculus} in $(\diamondsuit)$'s and Corollary~\ref{cor:split-quot-srg} in $(\clubsuit)$'s.
		The last step follows from $\srg{\resa{\opB} / \opI_{x, y}} = \{t\}$.
 \end{proof}

\section{Contraction factors of DYS}
	\label{s:dys_cont}
	Using the SRG theory of Section~\ref{s:srg_theory}, we prove new contraction factors of DYS in three settings of interest. The three contraction factors of Theorems~\ref{thm:lipschitz-dys-a-lip-sm}, \ref{thm:lipschitz-dys-a-lip-b-sm}, and \ref{thm:lipschitz-dys-a-lip-c-sm}
	are strictly smaller (better)%
 than the six contraction factors of \cite[Theorem~D.6]{davis2015three} and are, to the best of our knowledge, the best known rates.

	
	\begin{theorem}\label{thm:lipschitz-dys-a-lip-sm}
	    Let $0<\alpha < 4\beta_\cC<\infty$ and $\lambda \in \left( 0, 2 - \frac{\alpha}{2\beta_\cC}\right)$.
	    Let $\cH$ be a real Hilbert space.
	    Let $\opA\colon \cH \rightrightarrows \cH$ and $\opB\colon \cH \rightrightarrows \cH$ be monotone and $\opC \colon \cH \rightrightarrows \cH$ be $\beta_\cC$-cocoercive.
	    If $\opA$ is $\mu_\cA$-strongly monotone and $L_\cA$-Lipschitz for 
		$0 < \mu_\cA \le L_\cA < \infty$,
	    then $\opT_{\opA, \opB, \opC, \alpha, \lambda}$ is $\rho$-contractive with
	    \[
				\rho = 1 - \frac{2\lambda}{4 - \alpha / \beta_\cC} + \lambda \sqrt{ \frac{2}{4 - \alpha / \beta_\cC}  \left( \frac{2}{4 - \alpha / \beta_\cC} - \frac{2 \alpha \mu_\cA}{\alpha^2 L_\cA^2 + 2 \alpha \mu_\cA + 1} \right)}.
		\]
		Symmetrically, if $\opB$ is $\mu_\cB$-strongly monotone and $L_\cB$-Lipschitz for 
		$0 < \mu_\cB \le L_\cB < \infty$, then
		$\opT_{\opA, \opB, \opC, \alpha, \lambda}$ is $\rho$-contractive with
		\[
				\rho = 1 - \frac{2\lambda}{4 - \alpha / \beta_\cC} + \lambda \sqrt{ \frac{2}{4 - \alpha / \beta_\cC}  \left( \frac{2}{4 - \alpha / \beta_\cC} - \frac{2 \alpha \mu_\cB}{\alpha^2 L_\cB^2 + 2 \alpha \mu_\cB + 1} \right)}.
		\]
	\end{theorem}
	
	\begin{theorem}\label{thm:lipschitz-dys-a-lip-b-sm}
	    Let $0<\alpha < 4\beta_\cC<\infty$ and $\lambda \in \left( 0, 2 - \frac{\alpha}{2\beta_\cC}\right)$.
	    Let $\cH$ be a real Hilbert space.
	    Let $\opA\colon \cH \rightrightarrows \cH$ and $\opB\colon \cH \rightrightarrows \cH$ be monotone and $\opC \colon \cH \rightrightarrows \cH$ be $\beta_\cC$-cocoercive.
	    If $\opA$ is $L_\cA$-Lipschitz and $\opB$ is $\mu_\cB$-strongly monotone for $L_\cA,\mu_\cB\in (0,\infty)$, then 
	    $\opT_{\opA, \opB, \opC, \alpha, \lambda}$ is $\rho$-contractive with 
	    \[
			\rho = 
			\sqrt{1 - \min\left\{ 
				\frac{2 \alpha \mu_{\cB}\lambda(2-\lambda)}{(1+\alpha^2L_{\cA}^2)(2-\lambda+2\alpha \mu_{\cB})},
				\frac{2\alpha \lambda (( 2 - \lambda)(\mu_\cB+L_\cA) + 2\alpha \mu_\cB L_\cA) }
				{(1 + \alpha L_\cA)^2(2 - \lambda + 2 \alpha \mu_\cB)}
			\right\}}
			.
		\]
		Symmetrically, if $\opA$ is $\mu_\cA$-strongly monotone and $\opB$ is $L_\cB$-Lipschitz for $\mu_\cA,L_\cB\in (0,\infty)$, then $\opT_{\opA, \opB, \opC, \alpha, \lambda}$ is $\rho$-contractive with 
	    \[
			\rho = 
			\sqrt{1 - \min\left\{ 
				\frac{2 \alpha \mu_{\cA}\lambda(2-\lambda)}{(1+\alpha^2L_{\cB}^2)(2-\lambda+2\alpha \mu_{\cA})},
				\frac{2\alpha\lambda ( ( 2 - \lambda)(\mu_\cA+L_\cB) + 2\alpha\mu_\cA L_\cB) }
				{(1 + \alpha L_\cB)^2(2 - \lambda + 2 \alpha \mu_\cA)}
			\right\}}
			.
		\]
	\end{theorem}
	
	\begin{theorem}\label{thm:lipschitz-dys-a-lip-c-sm}
	    Let $0<\alpha < 4\beta_\cC<\infty$ and $\lambda \in \left( 0, 2 - \frac{\alpha}{2\beta_\cC}\right)$.
	    Let $\cH$ be a real Hilbert space.
	    Let $\opA\colon \cH \rightrightarrows \cH$ and $\opB\colon \cH \rightrightarrows \cH$ be monotone and $\opC \colon \cH \rightrightarrows \cH$ be $\beta_\cC$-cocoercive.
	    If $\opA$ is $L_\cA$-Lipschitz and $\opC$ is $\mu_\cC$-strongly monotone for $L_\cA\in (0,\infty)$ and $0 < \mu_\cC \le \frac{1}{\beta_\cC}$, then
	    $\opT_{\opA, \opB, \opC, \alpha, \lambda}$ is $\rho$-contractive with
	    \[
			\rho = 
			\sqrt{1 - 2\lambda\alpha\min\left\{ 
				\frac{L_\cA + \mu_\cC \left( 1 - \frac{\alpha}{2\beta_\cC(2-\lambda)} \right)}{(1 + \alpha L_\cA)^2}, 
				\frac{ \mu_\cC \left( 1 - \frac{\alpha}{2\beta_\cC(2-\lambda)} \right)}{1 + \alpha^2 L_\cA^2}
			\right\}}
			.
		\]
		Symmetrically, if $\opB$ is $L_\cB$-Lipschitz and $\opC$ is $\mu_\cC$-strongly monotone for $L_\cB\in (0,\infty)$ and $0 < \mu_\cC \le \frac{1}{\beta_\cC}$, then
	    $\opT_{\opA, \opB, \opC, \alpha, \lambda}$ is $\rho$-contractive with 
	    \[
			\rho = 
			\sqrt{1 - 2\lambda\alpha\min\left\{ 
				\frac{L_\cB + \mu_\cC \left( 1 - \frac{\alpha}{2\beta_\cC(2-\lambda)} \right)}{(1 + \alpha L_\cB)^2}, 
				\frac{ \mu_\cC \left( 1 - \frac{\alpha}{2\beta_\cC(2-\lambda)} \right)}{1 + \alpha^2 L_\cB^2}
			\right\}}
			.
		\]
	\end{theorem}

It is straightforward to verify that the values of $\rho$ in Theorems~\ref{thm:lipschitz-dys-a-lip-sm}, \ref{thm:lipschitz-dys-a-lip-b-sm}, and \ref{thm:lipschitz-dys-a-lip-c-sm}  indeed satisfy $\rho<1$, i.e., $\opT_{\opA, \opB, \opC, \alpha, \lambda}$ becomes a strict contraction, under the stated range of parameters.

	\subsection{Proofs of Theorems~\ref{thm:lipschitz-dys-a-lip-sm}, \ref{thm:lipschitz-dys-a-lip-b-sm}, and \ref{thm:lipschitz-dys-a-lip-c-sm}}

Before going through proofs, we introduce two elementary Lemmas.
\begin{lemma}
\label{lem:CS}
    For $a,b,c,d\in[0,\infty)$,
	\[
	\left(\sqrt{ab} + \sqrt{cd}\right)^2 \le (a + c)(b + d).
	\]
\end{lemma}
\begin{proof}
	This inequality is an instance of Cauchy--Schwarz.
	It can also be directly shown from
	 $\left(\sqrt{ab} + \sqrt{cd}\right)^2=(a + c)(b + d) - \left(\sqrt{ad} - \sqrt{bc}\right)^2$.    
\end{proof}
	
	\begin{lemma}\label{lem:distance-to-circ}
	Consider a circle with center $O$.
	Let $Q \neq O$ be a fixed point.
	($Q$ may otherwise be inside, outside, or on the circle.)
	Let $\phi\in [0,\pi]$, and let $P$ be a point on the circle such that $\angle POQ =\phi$.
	Then the length $\overline{QP}$ is an increasing function of $\phi\in [0,\pi]$.
	\end{lemma}

	\begin{center}
		\def\scale{2}
		\begin{tikzpicture}[scale=\scale]
			\def\lam{0.4};
			\def\t{120};
			\def\r{0.7};
			\def\a{0.1};

			\coordinate (O) at (\r, 0);
			\coordinate (A) at ({3*\r}, 0);
			\coordinate (P) at ({\r+\r*cos(\t)}, {\r*sin(\t)});
			
			\draw (O) circle ({\r});
			
			\filldraw (O) circle[radius={\ptsize / \scale pt}];
			\filldraw (A) circle[radius={\ptsize / \scale pt}];
			\filldraw (P) circle[radius={\ptsize / \scale pt}];
			
			\draw (O) node[below=0pt] {$O$};
			\draw (A) node[below right] {$Q$};
			\draw (P) node[left=4pt] {$P$};
			\draw (O) node[above right=3pt and -2pt] {$\phi$};
			
			\draw [very thick, densely dotted] (A) -- (P);
			\draw (A) -- (O);
			\draw (O) -- (P);
			
			\pic [draw, angle radius=\scale*\a*\cminpt] {angle=A--O--P};
		\end{tikzpicture}	
	\end{center}

	\begin{proof}
By the law of cosines, 
		$\overline{QP}^2 = \overline{PO}^2 + \overline{QO}^2 - 2\overline{PO}\,\overline{QO} \cos\phi$.
	\end{proof}

	\begin{proof}[Proof of Theorem~\ref{thm:lipschitz-dys-a-lip-sm}]
	Consider the first statement,
	     where $\opA$ is $\mu_\cA$-strongly monotone and $L_\cA$-Lipschitz.
		Write
		\[
		    \cA = \sm{\mu_\cA} \cap  \lip{L_\cA},\quad \cB = \cM,\quad \cC = \coco{\beta_\cC}.
		\]
		Then the first statement is equivalent to saying that $\opT_{\cA, \cB, \cC, \alpha, \lambda}$ is $\rho$-Lipschitz, where $\rho$ is as given.
		By Corollary~\ref{cor:tight-coeff-dys} and Lemma~\ref{lem:max-mod}, it is enough to show that
		$\sup_{z \in \cZ^{\text{DYS}}_{\cA, \cB, \cC, \alpha, \lambda}}$ $\abs{z}$
		is equal or smaller than $\rho$, i.e., it is enough to show 
		\[
		    |\zeta_\text{DYS}(z_A, z_B, z_C; \alpha, \lambda)| \le \rho
		\]
		for any $z_A \in\partial \srg{\resa{\cA}}, z_B \in \partial\srg{\resa{\cB}}$ and $z_C \in \partial\srg{\cC}$.

		Let $\theta = \frac{2}{4 - \alpha/\beta_\cC}$.  Then, 
		$\alpha z_C \in \alpha\partial\srg{\cC} = \Ci\left( 2 - \theta^{-1}, 2 - \theta^{-1} \right)$ by Fact~\ref{fact:monotone-srg}. So 
		$
			\abs{2 - \theta^{-1} -\alpha z_C} = 2 - \theta^{-1}.	
		$
		Therefore,
		\begin{align*}
			&\abs{\zeta_{\text{DYS}}(z_A, z_B, z_C; \alpha, \lambda)-(1-\lambda\theta)}^2\\
			&=\lambda^2\abs{\theta^{-1}(z_A-\theta)(z_B-\theta) +(2-\theta^{-1} -\alpha z_C) z_Az_B}^2\\
			&\le \lambda^2\left( \theta^{-1}\abs{z_A-\theta}\abs{z_B-\theta}+\abs{2-\theta^{-1} -\alpha z_C}\abs{z_A}\abs{z_B} \right)^2\\
			&= \lambda^2\left( \theta^{-1}\abs{z_A-\theta}\abs{z_B-\theta}+(2 - \theta^{-1})\abs{z_A}\abs{z_B} \right)^2\\
			&= \lambda^2\left( \sqrt{\theta^{-1}\abs{z_A-\theta}^2 \cdot \theta^{-1}\abs{z_B-\theta}^2} + \sqrt{(2 - \theta^{-1})\abs{z_A}^2 \cdot (2 - \theta^{-1})\abs{z_B}^2} \right)^2\\
			&\le \lambda^2\left( \theta^{-1} \abs{z_A-\theta}^2+(2- \theta^{-1})\abs{z_A}^2 \right) \left( \theta^{-1} \abs{z_B-\theta}^2+(2- \theta^{-1})\abs{z_B}^2 \right),
		\end{align*}
		where the first inequality follows from the triangle inequality and the second is from Lemma~\ref{lem:CS}.
		The second factor simplifies to
		\[
			\theta^{-1} \abs{z_B - \theta}^2 + (2 - \theta^{-1}) \abs{z_B}^2
			= 2\left| z_B - \frac{1}{2} \right|^2 + \theta - \frac{1}{2}
			= \theta,
		\]
		since $z_B \in \partial\srg{\resa{\cB}} = \Ci\left( \frac{1}{2}, \frac{1}{2} \right)$ by Fact~\ref{fact:monotone-srg} and~\ref{fact:srg-trans}.

		We now bound the first factor. Similarly,
		\begin{align*}
			\theta^{-1} \abs{z_A - \theta}^2 + (2 - \theta^{-1}) \abs{z_A}^2
			= 2\left| z_A - \frac{1}{2} \right|^2 + \theta - \frac{1}{2}, 
		\end{align*}
		which is maximized (worst case) when $\left| z_A - \frac{1}{2} \right|$ is maximized.
  
	Let 
	\[
		O_\mu = \frac{1}{2(1+\alpha\mu_{\cA})},
		\quad
		O_L = \frac{1}{1-\alpha^2L_{\cA}^2}
	\]
	be the center of $\partial \srg{\resa{\sm{\mu_{\cA}}}}$ and the center of $\partial \srg{\resa{\lip{L_{\cA}}}}$ (or $\infty$ if $\alpha L_\cA = 1$), respectively.
	Let 
	\[
		P = z_A,
		\quad
        Q = \frac{1}{2}, 
		\quad	
		\left\{ N, N' \right\} = \partial \srg{\resa{\lip{L_{\cA}}}} \cap \partial \srg{\resa{\sm{\mu_{\cA}}}}.
	\]
	We now characterize the maximum value of $\overline{QP} = \abs{z_A - \frac{1}{2}}$ for each of three cases depicted in the figure below. (SRG of resolvent can be easily drawn by inversive geometry introduced in section 2 of \cite{ryu_scaled_2021}.)  
    \begin{center}
		\begin{tabular}{ccc}
		\def\scale{2.2}
		\begin{tikzpicture}[scale=\scale]
		
		\def\m{0.5};
		\def\L{3};
		\def\a{1};
		
		\def\c{20};
		\def\A{0.5};
		
		\def\LipR{\a*\L/(\a^2*\L^2-1)};
		\def\MuR{1/(2*(1+\a*\m))};
		
		\coordinate (C) at ({(1+\a*\m)/(\a^2*\L^2+2*\a*\m+1)},{(\a*sqrt(\L^2-\m^2))/(\a^2*\L^2+2*\a*\m+1)});   
		\coordinate (C') at ({(1+\a*\m)/(\a^2*\L^2+2*\a*\m+1)},{-(\a*sqrt(\L^2-\m^2))/(\a^2*\L^2+2*\a*\m+1)});   
		\coordinate (O') at ({1/(2*(1+\a*\m))},0);
		\coordinate (O) at ({1/(1-\a^2*\L^2)},0);
		\coordinate (Z) at (0,0);
		\coordinate (P) at ({1/(1-\a^2*\L^2) + \LipR * cos(\c)},{\LipR * sin(\c)});
		\coordinate (A) at (\A,0);
		\coordinate (S) at ({(1+\a*\m)/(\a^2*\L^2+2*\a*\m+1)+0.05},{(\a*sqrt(\L^2-\m^2))/(\a^2*\L^2+2*\a*\m+1)+0.17});   
		
		\fill[fill=medgrey] (O') circle ({\MuR});
		\fill[fill=white] (O) circle ({\LipR});
		
		\pic [draw, dashed, angle radius=\scale*\MuR*\cminpt] {angle=C--O'--C'};
		\pic [draw, very thick, angle radius=\scale*\MuR*\cminpt] {angle=C'--O'--C};
		\pic [draw, dashed, angle radius=\scale*\LipR*\cminpt] {angle=C--O--C'};
		\pic [draw, very thick, angle radius=\scale*\LipR*\cminpt] {angle=C'--O--C};

		\filldraw (P) circle[radius={\ptsize / \scale pt}];
		\filldraw (C) circle[radius={\ptsize / \scale pt}];
		\filldraw (C') circle[radius={\ptsize / \scale pt}];
		\filldraw (O) circle[radius={\ptsize / \scale pt}];
		\filldraw (O') circle[radius={\ptsize / \scale pt}];
		\filldraw (A) circle[radius={\ptsize / \scale pt}];
		
		\draw (S) node[right=3pt] {$\srg{\resa{\cA}}$};
		\draw (P) node[left] {$P$};
		\draw (A) node[below right=0pt and -4pt] {$Q$};
		\draw (O) node[below] {$O_L$};
		\draw (O') node[below right=0pt and -6pt] {$O_\mu$};
		\draw (C) node[above] {$N$};
		\draw (C') node[below] {$N'$};
		\draw [<->] (-0.8,0) -- (0.8,0);
		\draw [<->] (0,-0.6) -- (0,0.6);
		\draw [very thick, densely dotted] (A) -- (P);
		\draw (0.1,-0.73) node[below, align=left] {Case $\alpha L_\cA > 1$};
		\end{tikzpicture}
		&
		\def\scale{1.8}
		\begin{tikzpicture}[scale=\scale]
		
		\def\m{0.2};
		\def\L{0.5};
		\def\a{1};
		
		\def\c{170};
		\def\A{0.5};
		
		\def\LipR{\a*\L/(1-\a^2*\L^2)};
		\def\MuR{1/(2*(1+\a*\m))};
		
		\coordinate (C) at ({(1+\a*\m)/(\a^2*\L^2+2*\a*\m+1)},{(\a*sqrt(\L^2-\m^2))/(\a^2*\L^2+2*\a*\m+1)});   
		\coordinate (C') at ({(1+\a*\m)/(\a^2*\L^2+2*\a*\m+1)},{-(\a*sqrt(\L^2-\m^2))/(\a^2*\L^2+2*\a*\m+1)});   
		\coordinate (O') at ({1/(2*(1+\a*\m))},0);
		\coordinate (O) at ({1/(1-\a^2*\L^2)},0);
		\coordinate (Z) at (0,0);
		\coordinate (P) at ({1/(1-\a^2*\L^2) + \LipR * cos(\c)},{\LipR * sin(\c)});
		\coordinate (A) at (\A,0);
		
		\begin{scope}
		\clip (O) circle ({\LipR});
		\fill[fill=medgrey] (O') circle ({\MuR});
		\end{scope}

		\pic [draw, dashed, angle radius=\scale*\LipR*\cminpt] {angle=C'--O--C};
		\pic [draw, very thick, angle radius=\scale*\LipR*\cminpt] {angle=C--O--C'};
		\pic [draw, dashed, angle radius=\scale*\MuR*\cminpt] {angle=C--O'--C'};
		\pic [draw, very thick, angle radius=\scale*\MuR*\cminpt] {angle=C'--O'--C};

		\filldraw (P) circle[radius={\ptsize / \scale pt}];
		\filldraw (C) circle[radius={\ptsize / \scale pt}];
		\filldraw (C') circle[radius={\ptsize / \scale pt}];
		\filldraw (O') circle[radius={\ptsize / \scale pt}];
		\filldraw (O) circle[radius={\ptsize / \scale pt}];
		\filldraw (A) circle[radius={\ptsize / \scale pt}];
		
		\draw (C) node[right=3pt] {$\srg{\resa{\cA}}$};
		\draw (P) node[above left=-3pt and -3pt] {$P$};
		\draw (A) node[below] {$Q$};
		\draw (O') node[below left=0pt and -3pt] {$O_\mu$};
		\draw (O) node[below] {$O_L$};
		\draw (C) node[above=4pt] {$N$};
		\draw (C') node[below=4pt] {$N'$};
		\draw [<->] (-0.2,0) -- (2.2,0);
		\draw [<->] (0,-0.75) -- (0,0.75);
		\draw [very thick, densely dotted] (A) -- (P);
		\draw (1,-0.88) node[below, align=left] {Case $\alpha L_{\cA}<1$};
		\end{tikzpicture}
		&
		\def\scale{1.4}
		\begin{tikzpicture}[scale=\scale]
		\def\a{0.85};
		\def\b{1.1};
		\def\c{0.3};
		\def\A{0.6};
		\def\t{20};
		
		\coordinate (C) at ({0.5},{sqrt(1-2*\a+\b^2-0.25)});   
		\coordinate (C') at ({0.5},{-sqrt(1-2*\a+\b^2-0.25)});    
		\coordinate (O') at (0,0);
		\coordinate (P) at ({sqrt(1-2*\a+\b^2)*cos(\t)},{sqrt(1-2*\a+\b^2)*sin(\t)});
		\coordinate (A) at (0.5,0);
		\coordinate (Z) at ({-sqrt(1-2*\a+\b^2)},0);
		
		\begin{scope}
			\clip ({0.5},{sqrt(1-2*\a+\b^2-0.25)}) rectangle (1,-1);
			\fill[fill=medgrey] (O') circle ({sqrt(1-2*\a+\b^2)});
		\end{scope}
		
		\pic [draw, dashed, angle radius=\scale*sqrt(1-2*\a+\b^2)*\cminpt] {angle=C--O'--C'};
		\pic [draw, very thick, angle radius=\scale*sqrt(1-2*\a+\b^2)*\cminpt] {angle=C'--O'--C};
		\draw [dashed] (C) -- ({0.5},0.8);
		\draw [dashed] (C') -- ({0.5},-0.8);
		\draw [very thick] (C) -- (C');

		\filldraw (P) circle[radius={\ptsize / \scale pt}];
		\filldraw (C') circle[radius={\ptsize / \scale pt}];
		\filldraw (C) circle[radius={\ptsize / \scale pt}];
		\filldraw (O') circle[radius={\ptsize / \scale pt}];
		\filldraw (A) circle[radius={\ptsize / \scale pt}];
		
		\draw (C) node[below left] {$\srg{\resa{\cA}}$};
		\draw (P) node[right] {$P$};
		\draw (A) node[below left] {$Q$};
		\draw (O') node[below] {$O_\mu$};
		\draw (C) node[above right] {$N$};
		\draw (C') node[below right] {$N'$};
		
		\draw [<->] (-1,0) -- (1,0);
		\draw [<->] ({-sqrt(1-2*\a+\b^2)},-0.95) -- ({-sqrt(1-2*\a+\b^2)},0.95);
		\draw [very thick, densely dotted] (A) -- (P);
		\draw (0.3,-1.3) node[fill=white] {Case $\alpha L_{\cA}= 1$};
		\end{tikzpicture}
		\end{tabular}
	\end{center}
    If $\alpha L_{\cA} > 1$, then
	$
		\frac{1}{1-\alpha^2 L_{\cA}^2} < 0 < \frac{1}{2(1+\alpha \mu_{\cA})} < \frac{1}{2}.
	$
	So $Q$ is on the right side of both centers of circles, $O_L$ and $O_\mu$.  
	If $\alpha L_{\cA} < 1$, then
	$
		\frac{1}{2(1+\alpha \mu_{\cA})} < \frac{1}{2} < 1 < \frac{1}{1-\alpha^2 L_{\cA}^2}.
	$
	So $Q$ resides between $O_L$ and $O_\mu$. 
	In both cases, $\overline{QP}$ attains its maximum at $P = N$ or $N'$  by Lemma~\ref{lem:distance-to-circ}.
	If $\alpha L_{\cA} = 1$, a similar reasoning shows that $\overline{QP}$ attains its maximum at $P = N$ or $N'$ by Lemma~\ref{lem:distance-to-circ}.

	Thus,
	\[
		\argmax_{z_A \in \partial\srg{\resa{\cA}}} \, \babs{z_A-\frac{1}{2}}
		=\{N, N'\}
		= \partial\srg{\resa{\lip{L_\cA}}} \cap \partial\srg{\resa{\sm{\mu_\cA}}}.
	\]
Since
	\begin{align*}
	    \{N,N'\}
		&= \frac{1}{\alpha^2 L_\cA^2 + 2\alpha \mu_\cA + 1}
		\left( 
		    \alpha\mu_\cA + 1 \pm i \alpha \sqrt{L_\cA^2 - \mu_\cA^2}
		\right),
	\end{align*}
and
\begin{align*}
		\max_{z_A \in \partial\srg{\resa{\cA}}} \, \babs{z_A-\frac{1}{2}}^2
		&=
		\left( \frac{\alpha\mu_\cA + 1}{\alpha^2 L_\cA^2 + 2\alpha\mu_\cA + 1} - \frac{1}{2} \right)^2
		+
		\left( \frac{\alpha \sqrt{L_\cA^2 - \mu_\cA^2}}{\alpha^2 L_\cA^2 + 2\alpha\mu_\cA + 1} \right)^2
		\\
		&=
		\frac{1}{4} - \frac{\alpha\mu_\cA}{\alpha^2 L_\cA^2 + 2\alpha\mu_\cA + 1},
	\end{align*}
we bound the first factor as
    \[
        \theta^{-1}\abs{z_A - \theta}^2 + \abs{2 - \theta^{-1} -\alpha z_C}\abs{z_A}^2 \le \theta - \frac{2\alpha\mu_\cA}{\alpha^2 L_\cA^2 + 2 \alpha \mu_\cA + 1}.
    \]
		
        Finally, we conclude
		\begin{align*}
			\abs{\zeta_{\text{DYS}}(z_A, z_B, z_C; \alpha, \lambda)}
			&\le \abs{1 - \lambda\theta} + \abs{\zeta_{\text{DYS}}(z_A, z_B, z_C; \alpha, \lambda) - (1-\lambda\theta)}\\
			&\le 1 - \lambda\theta + \lambda\sqrt{\theta \left( \theta - \frac{2\alpha\mu_\cA}{\alpha^2 L_\cA^2 + 2 \alpha \mu_\cA + 1} \right)}.
		\end{align*}
		Note that $\lambda < 2 - \frac{\alpha}{2\beta_\cC} = \theta^{-1}$. So $1 - \lambda \theta > 0$. This completes the proof for the first statement. 
		
		Since $\zeta_{\text{DYS}}(z_A, z_B, z_C; \alpha, \lambda)$ is symmetric with respect to $z_A$ and $z_B$, the second statement follows from the same argument.
	\end{proof}


	\begin{proof}[Proof of Theorem~\ref{thm:lipschitz-dys-a-lip-b-sm}]
	Consider the first statement,
	     where $\opA$ is $L_\cA$-Lipschitz and $\opB$ is $\mu_\cB$-strongly monotone.
		Write
		\[
			\cA = \cM \cap \lip{L_\cA}, \quad \cB = \sm{\mu_\cB}, \quad \cC = \coco{\beta_\cC}.
		\]
		By the same reasoning as for Theorem~\ref{thm:lipschitz-dys-a-lip-sm}, it suffices to show
		\[
		    \left| \zeta_{\text{DYS}}(z_A, z_B, z_C; \alpha, \lambda) \right| \le \rho
		\]
		for any $z_A \in \partial\srg{\resa{\cA}}$, $z_B \in \partial\srg{\resa{\cB}}$, and $z_C \in \partial\srg{\cC}$.
		
		Let 
		\[
			O_\beta = \frac{\alpha}{2\beta_\cC}, 
			\quad P = \alpha z_C,
            \quad Q = 2 - \lambda, 
            \quad Z = 0,
		\]
        as depicted in the following figure.
        \begin{center}
			\def\scale{2}
			\begin{tikzpicture}[scale=\scale]
				\def\a{1};
				\def\b{0.7};
				\def\lam{0.4};
				\def\t{130};
				\def\k{300};
				\def\e{1.6};

				\coordinate (Z) at (0, 0);
				\coordinate (O) at (\a/\b/2, 0);
				\coordinate (A) at (2-\lam, 0);
				\coordinate (P) at ({\a/\b/2+\a/\b/2*cos(\t)}, {\a/\b/2*sin(\t)});
				\coordinate (L) at ({\a/\b/2+\a/\b/2*cos(\k)}, {\a/\b/2*sin(\k)});
				
				\fill[fill=medgrey] (O) circle ({\a/\b/2});
				\draw[very thick] (O) circle ({\a/\b/2});

				\filldraw (Z) circle[radius={\ptsize / \scale pt}];
				\filldraw (O) circle[radius={\ptsize / \scale pt}];
				\filldraw (A) circle[radius={\ptsize / \scale pt}];
				\filldraw (P) circle[radius={\ptsize / \scale pt}];

				\draw (Z) node[below left=0pt] {$Z$};
				\draw (O) node[below right=0pt and -5pt] {$O_\beta$};
				\draw (A) node[below right=0pt] {$Q$};
				\draw (P) node[above left=-4pt] {$P$};
				\draw (L) node[below right=0pt] {$\alpha\srg{\cC}$};
				
				\draw [<->] (-1.2,0) -- (2.2,0);
				\draw [<->] (0,-1.2) -- (0,1.2);
				\draw [very thick, densely dotted] (A) -- (P);
			\end{tikzpicture}	
		\end{center}		
		Observe that 
$\alpha z_C \in \Ci\left( \frac{\alpha}{2\beta_\cC}, \frac{\alpha}{2\beta_\cC} \right)$.
		Since we assume $2 - \lambda > \frac{\alpha}{2\beta_\cC}$, 
		$Q$ is on the right side of $O_\beta$. Thus,
		$
			\abs{2 - \lambda - \alpha z_C} = \overline{QP}
			\le \overline{QZ} = 2 - \lambda
		$
		by Lemma~\ref{lem:distance-to-circ}.
		Then, we have
		\begin{align*}
    		&\left| \zeta_{\text{DYS}}(z_A, z_B, z_C; \alpha, \lambda) \right|^2 \\
    		&= \left| (1 - \lambda z_A)(1 - \lambda z_B) + \lambda(2 - \lambda - \alpha z_C) z_A z_B \right|^2 \\
    		&\le \left( \left| (1 - \lambda z_A)(1 - \lambda z_B) \right| + 
    		\lambda \left| 2 - \lambda - \alpha z_C\right| \left| z_A z_B \right|
    		\right)^2 \\
			&\le (\abs{1 - \lambda z_A}\abs{1 - \lambda z_B} + \lambda(2 - \lambda)\abs{z_A}\abs{z_B})^2\\
			&= \left(
			\sqrt{\abs{1 - \lambda z_A}^2 \abs{1 - \lambda z_B}^2}	
			+
			\sqrt{
				\left( r^{-1}\lambda(2 - \lambda)\abs{z_A}^2 \right)
				\left( r\lambda(2 - \lambda)\abs{z_B}^2 \right)
			}
			\right)^2\\
    		&\le \left( \left| 1 - \lambda z_A \right|^2 + r^{-1} \lambda (2 - \lambda) \left| z_A \right|^2 \right) 
    		\left( \left| 1 - \lambda z_B \right|^2 + r \lambda (2 - \lambda) \left| z_B \right|^2 \right)
    	\end{align*}
    	for any $r>0$,
    	where the first inequality follows from the triangle inequality, the second is from $\overline{QP}
			\le \overline{QZ} $, and the third is from Lemma~\ref{lem:CS}.
    	Set 
    	\[
    		r = \frac{2 - \lambda + 2 \alpha \mu_\cB}{2 - \lambda}.
    	\]
    	The second factor simplifies to
        \[
        \abs{1 - \lambda z_B}^2 + r\lambda(2 - \lambda)\abs{z_B}^2 
    		= 2\lambda(1 + \alpha \mu_\cB)\left| z_B - \frac{1}{2(1 + \alpha \mu_\cB)}\right|^2 + 1 - \frac{\lambda}{2(1 + \alpha \mu_\cB)}
    		= 1
        \]
    	since $z_B \in \partial\srg{\resa{\cB}} = \Ci\left( \frac{1}{2(1+\alpha \mu_{\cB})}, \frac{1}{2(1+\alpha \mu_{\cB})} \right)$. Therefore,
    	\begin{align*}
    		&\left| \zeta_\text{DYS}(z_A, z_B, z_C; \alpha, \lambda) \right|^2\\
    		&\le \left| 1 - \lambda z_A \right|^2 + r^{-1} \lambda (2 - \lambda) \left| z_A \right|^2\\
    		&= \lambda (\lambda + (2 - \lambda)r^{-1})
    		\left| z_A - \frac{1}{\lambda + (2 - \lambda)r^{-1}}\right|^2 
    		+ 1 - \frac{\lambda}{\lambda + (2 - \lambda)r^{-1}},
    	\end{align*}
    	which is maximized when
		$\left| z_A - \frac{1}{\lambda + (2 - \lambda)r^{-1}} \right|$ is maximized.

	Let
	\begin{gather*}
			O_0 = \frac{1}{2}, \quad 
			O_L = \frac{1}{1 - \alpha^2 L_{\cA}^2}, \quad
			P = z_A, \quad
            Q = \frac{1}{\lambda + (2-\lambda)r^{-1}},
			\\
			M = \frac{1}{1 + \alpha L_\cA}, \quad
			\left\{ N, N' \right\} = \partial\cG\big(\resa{\lip{L_\cA}}\big) \cap \partial\cG\big(\resa{\cM}\big) = \frac{1}{1 + \alpha^2 L_\cA^2}\left( 1 \pm i \alpha L_\cA \right).
    \end{gather*}
    We now characterize the maximum value of $\overline{QP} = \abs{z_A - \frac{1}{\lambda + (2 - \lambda)r^{-1}}}$.
    First, note that $Q$ is on the right side of $O_0$, since $r = \frac{2 - \lambda + 2 \alpha \mu_\cB}{2 - \lambda} > 1$ and thus $Q = \frac{1}{\lambda + (2 - \lambda)r^{-1}} > \frac{1}{2} = O_0$.
    Four cases in the following figure cover all possibilities (including the case $O_0 \le Q \le M$, in particular), and we can confirm that $\overline{QP}$ attains its maximum at $P = M$ or $P \in \{N, N'\}$ in each case by applying Lemma~\ref{lem:distance-to-circ} together with the fact that $Q$ is on the right side of $O_0$.
    \begin{center}
			\begin{tabular}{cc}
				{
					\def\scale{1.4}
					\begin{tikzpicture}[scale=\scale]
						\def\a{1};
						\def\L{0.6};
						\def\b{0.7};
						\def\lam{0.4};
						\def\t{170};
						\def\k{300};
						\def\e{1.6};
						\def\m{0.35};
		
						\coordinate (Z) at (0, 0);
						\coordinate (O) at ({1/(1-\a^2*\L^2)}, {0});
						\coordinate (O') at (1/2, 0);
						\coordinate (A) at (1.3, 0);
						\coordinate (P) at ({1/(1-\a^2*\L^2)+\a*\L/(1-\a^2*\L^2)*cos(\t)}, {\a*\L/(1-\a^2*\L^2)*sin(\t)});
						\coordinate (C) at ({1/(1+\a^2*\L^2)}, {\a*\L/(1+\a^2*\L^2)});
						\coordinate (C') at ({1/(1+\a^2*\L^2)}, {-\a*\L/(1+\a^2*\L^2)});
						\coordinate (M) at ({1/(1+\a*\L)}, 0);
						
						\begin{scope}
							\clip (O') circle ({1/2});
							\fill[fill=medgrey] (O) circle ({\a*\L/(1-\a^2*\L^2)});
						\end{scope}
		

						\filldraw (O) circle[radius={\ptsize / \scale pt}];
						\filldraw (O') circle[radius={\ptsize / \scale pt}];
						\filldraw (A) circle[radius={\ptsize / \scale pt}];
						\filldraw (C) circle[radius={\ptsize / \scale pt}];
						\filldraw (C') circle[radius={\ptsize / \scale pt}];
						\filldraw (P) circle[radius={\ptsize / \scale pt}];
		
						\draw (O) node[below right=0pt] {$O_L$};
						\draw (O') node[below left=0pt] {$O_0$};
						\draw (A) node[below=-1pt] {$Q$};
						\draw (P) node[left=-1pt] {$P$};
						\draw (C) node[above left=3pt and -3pt] {$N$};
						\draw (C') node[below left=3pt and -3pt] {$N'$};
						\draw (C) node[right=3pt] {$\srg{\resa{\cA}}$};
		
						\pic [draw, dashed, angle radius={\scale*\cminpt*(\a*\L/(1-\a^2*\L^2))}] {angle=C'--O--C};
						\pic [draw, very thick, angle radius={\scale*\cminpt*(\a*\L/(1-\a^2*\L^2))}] {angle=C--O--C'};
						\pic [draw, dashed, angle radius={\scale*\cminpt*(1/2)}] {angle=C--O'--C'};
						\pic [draw, very thick, angle radius={\scale*\cminpt*(1/2)}] {angle=C'--O'--C};
						
						\draw [<->] (-0.3,0) -- (2.7,0);
						\draw [<->] (0,-1.2) -- (0,1.2);
						\draw [very thick, densely dotted] (A) -- (P);
						\draw (1,-1.3) node[below, align=left] {Case $\alpha L_\cA < 1$, $Q$ is on the left of $O_L$:
						\\ $P = z_A \in \frac{1}{1 + \alpha^2 L_\cA^2}(1 \pm i \alpha L_\cA) = \left\{ N, N' \right\}$ 
						\\ are maxima.};
					\end{tikzpicture}	
				}
				&
				\raisebox{0\height}{
					\def\scale{1.4}
					\begin{tikzpicture}[scale=\scale]
						\def\a{1};
						\def\L{0.6};
						\def\b{0.7};
						\def\lam{0.4};
						\def\t{170};
						\def\k{300};
						\def\e{1.6};
						\def\m{0.35};
		
						\coordinate (Z) at (0, 0);
						\coordinate (O) at ({1/(1-\a^2*\L^2)}, {0});
						\coordinate (O') at (1/2, 0);
						\coordinate (A) at (2, 0);
						\coordinate (P) at ({1/(1-\a^2*\L^2)+\a*\L/(1-\a^2*\L^2)*cos(\t)}, {\a*\L/(1-\a^2*\L^2)*sin(\t)});
						\coordinate (C) at ({1/(1+\a^2*\L^2)}, {\a*\L/(1+\a^2*\L^2)});
						\coordinate (C') at ({1/(1+\a^2*\L^2)}, {-\a*\L/(1+\a^2*\L^2)});
						\coordinate (M) at ({1/(1+\a*\L)}, 0);
						
						\begin{scope}
							\clip (O') circle ({1/2});
							\fill[fill=medgrey] (O) circle ({\a*\L/(1-\a^2*\L^2)});
						\end{scope}
		

						\filldraw (O) circle[radius={\ptsize / \scale pt}];
						\filldraw (O') circle[radius={\ptsize / \scale pt}];
						\filldraw (A) circle[radius={\ptsize / \scale pt}];
						\filldraw (M) circle[radius={\ptsize / \scale pt}];
						\filldraw (P) circle[radius={\ptsize / \scale pt}];

						\draw (O) node[below=0pt] {$O_L$};
						\draw (O') node[below left=0pt and -5pt] {$O_0$};
						\draw (A) node[below right=-1pt] {$Q$};
						\draw (P) node[left=-1pt] {$P$};
						\draw (C) node[right=3pt] {$\srg{\resa{\cA}}$};
						\draw (M) node[below right=0pt and -1pt] {$M$};
		
						\pic [draw, dashed, angle radius={\scale*\cminpt*(\a*\L/(1-\a^2*\L^2))}] {angle=C'--O--C};
						\pic [draw, very thick, angle radius={\scale*\cminpt*(\a*\L/(1-\a^2*\L^2))}] {angle=C--O--C'};
						\pic [draw, dashed, angle radius={\scale*\cminpt*(1/2)}] {angle=C--O'--C'};
						\pic [draw, very thick, angle radius={\scale*\cminpt*(1/2)}] {angle=C'--O'--C};

						\draw [<->] (-0.3,0) -- (2.7,0);
						\draw [<->] (0,-1.2) -- (0,1.2);
						\draw [very thick, densely dotted] (A) -- (P);
						\draw (1,-1.2) node[below=5pt, align=left] {
						Case $\alpha L_\cA < 1$, $Q$ is on the right of \\ $O_L$ : 
						$P = z_A = \frac{1}{1 + \alpha L_\cA} = M$ is\\maxima.};
					\end{tikzpicture}	
				}
			\end{tabular}
			
			\begin{tabular}{cc}
		    	{
				\def\scale{1.4}
				\begin{tikzpicture}[scale=\scale]
					\def\a{1};
					\def\L{1.5};
					\def\b{0.7};
					\def\lam{0.4};
					\def\t{50};
					\def\k{300};
					\def\e{1.6};
					\def\m{0.35};
	
					\coordinate (Z) at (0, 0);
					\coordinate (O) at ({1/(1-\a^2*\L^2)}, {0});
					\coordinate (O') at (1/2, 0);
					\coordinate (A) at (1.3, 0);
					\coordinate (P) at ({1/2*cos(\t)+1/2},{1/2*sin(\t)});
					\coordinate (C) at ({1/(1+\a^2*\L^2)}, {\a*\L/(1+\a^2*\L^2)});
					\coordinate (C') at ({1/(1+\a^2*\L^2)}, {-\a*\L/(1+\a^2*\L^2)});
					\coordinate (L) at ({1/2},{-1/2});
					\coordinate (M) at ({1/(1+\a*\L)}, 0);
					
					\begin{scope}
						\fill[fill=medgrey] (O') circle ({1/2});
						\fill[fill=white] (O) circle ({\a*\L/(1-\a^2*\L^2)});
					\end{scope}
	

					\filldraw (O) circle[radius={\ptsize / \scale pt}];
					\filldraw (O') circle[radius={\ptsize / \scale pt}];
					\filldraw (A) circle[radius={\ptsize / \scale pt}];
					\filldraw (P) circle[radius={\ptsize / \scale pt}];
					\filldraw (C) circle[radius={\ptsize / \scale pt}];
					\filldraw (C') circle[radius={\ptsize / \scale pt}];
	
					\draw (O) node[below right=0pt] {$O_L$};
					\draw (O') node[below right=0pt] {$O_0$};
					\draw (A) node[below=-1pt] {$Q$};
					\draw (P) node[above right=-2pt and -2pt] {$P$};
					\draw (C) node[above right=0pt and -3pt] {$N$};
					\draw (C') node[below right=0pt and -3pt] {$N'$};
					\draw (L) node[below right=3pt] {$\srg{\resa{\cA}}$};
	
					\pic [draw, dashed, angle radius={\scale*\cminpt*(\a*\L/(\a^2*\L^2-1))}] {angle=C--O--C'};
					\pic [draw, dashed, angle radius={\scale*\cminpt*(1/2)}] {angle=C--O'--C'};
					\pic [draw, very thick, angle radius={\scale*\cminpt*(\a*\L/(\a^2*\L^2-1))}] {angle=C'--O--C};
					\pic [draw, very thick, angle radius={\scale*\cminpt*(1/2)}] {angle=C'--O'--C};
					
					\draw [<->] (-2.3,0) -- (1.5,0);
					\draw [<->] (0,-1.2) -- (0,1.2);
					\draw [very thick, densely dotted] (A) -- (P);
					\draw (-0.2,-1.3) node[below, align=left] {
					Case $\alpha L_\cA > 1$ : \\
					$P = z_A \in \frac{1}{1 + \alpha^2 L_\cA^2}(1 \pm i \alpha L_\cA) =$\\ $ \left\{ N, N' \right\}$ are maxima.};
				\end{tikzpicture}	
			    }
			     &
			     {
				\def\scale{1.4}
				\begin{tikzpicture}[scale=\scale]
					\def\b{0.7};
					\def\lam{0.4};
					\def\t{50};
					\def\k{300};
					\def\e{1.6};
					\def\m{0.35};
	
					\coordinate (Z) at (0, 0);
					\coordinate (O') at (1/2, 0);
					\coordinate (A) at (1.5, 0);
					\coordinate (P) at ({1/2*cos(\t)+1/2},{1/2*sin(\t)});
					\coordinate (C) at ({1/2}, {1/2});
					\coordinate (C') at ({1/2}, {-1/2});
					\coordinate (L) at ({1/2+1/2*cos(\k)},{1/2*sin(\k)});
					
					\begin{scope}
					    \clip (0.5,-0.5) rectangle (1,0.5);
						\fill[fill=medgrey] (O') circle ({1/2});
					\end{scope}
					
					\filldraw (O') circle[radius={\ptsize / \scale pt}];
					\filldraw (A) circle[radius={\ptsize / \scale pt}];
					\filldraw (P) circle[radius={\ptsize / \scale pt}];
					\filldraw (C) circle[radius={\ptsize / \scale pt}];
					\filldraw (C') circle[radius={\ptsize / \scale pt}];
	
					\draw (O') node[below right=0pt] {$O_0$};
					\draw (A) node[below=-1pt] {$Q$};
					\draw (P) node[above right=-2pt and -2pt] {$P$};
					\draw (C) node[above left=0pt and 0pt] {$N$};
					\draw (C') node[below left=0pt and 0pt] {$N'$};
					\draw (L) node[below right=1.5pt] {$\srg{\resa{\cA}}$};
	
					\pic [draw, dashed, angle radius={\scale*\cminpt*(1/2)}] {angle=C--O'--C'};
					\pic [draw, very thick, angle radius={\scale*\cminpt*(1/2)}] {angle=C'--O'--C};
					\draw [very thick] (C) -- (C');
					
					\draw [dashed] (0.5,-1.2) -- (0.5,1.2);
					
					\draw [<->] (-0.4,0) -- (1.7,0);
					\draw [<->] (0,-1.2) -- (0,1.2);
					\draw [very thick, densely dotted] (A) -- (P);
					\draw (-0.2,-1.3) node[below right=0pt and -33pt, align=left] {
					Case $\alpha L_\cA = 1$ : \\
					$P = z_A \in \frac{1}{1 + \alpha^2 L_\cA^2}(1 \pm i \alpha L_\cA) =$ \\ $\left\{ N, N' \right\}$ are maxima.};
				\end{tikzpicture}	
			}
			\end{tabular}
		\end{center}
        If $\overline{QP}$ attains its maximum at $P = M$, i.e., $z_A = \frac{1}{1+\alpha L_\cA}$,
    	\begin{align*}
    		\abs{1-\lambda z_A}^2+r^{-1}\lambda(2 - \lambda)\abs{z_A}^2
    	    = 1 -  \frac{2\alpha\lambda ( \mu_\cB ( 2 - \lambda) + 
    	     L_\cA(2 - \lambda + 2\alpha \mu_\cB) )}
    		{(1 + \alpha L_\cA)^2(2 - \lambda + 2 \alpha \mu_\cB)}.
    	\end{align*}
    	Otherwise, if $\overline{QP}$ attains its maximum at $P \in \{N, N'\}$, i.e., $z_A \in \frac{1}{1 + \alpha^2 L_\cA^2}(1 \pm i \alpha L_\cA)$, 
    	\begin{align*}
    	    \abs{1-\lambda z_A}^2+r^{-1}\lambda(2 - \lambda)\abs{z_A}^2
    	    = 1-\frac{2 \alpha \mu_{\cB}\lambda(2-\lambda)}{(1+\alpha^2L_{\cA}^2)(2-\lambda+2\alpha \mu_{\cB})}.
    	\end{align*}
    	Therefore,
    	\begin{align*}
    		&\abs{\zeta_{\text{DYS}}(z_A, z_B, z_C; \alpha, \lambda)}^2 \\
			&\le
    			1 - \min\left\{ 
    			\frac{2 \alpha \mu_{\cB}\lambda(2-\lambda)}{(1+\alpha^2L_{\cA}^2)(2-\lambda+2\alpha \mu_{\cB})},
    			\frac{2\alpha\lambda ( ( 2 - \lambda)(\mu_\cB +L_\cA) + 
    			 2\alpha \mu_\cB L_\cA) }
    			{(1 + \alpha L_\cA)^2(2 - \lambda + 2 \alpha \mu_\cB)}
    		 \right\}.
    	\end{align*}
    	This completes the proof for the first statement. 
		
        Since $\zeta_{\text{DYS}}(z_A, z_B, z_C; \alpha, \lambda)$ is symmetric with respect to $z_A$ and $z_B$, the second statement follows from the same argument.
        \end{proof}
    
    	\begin{proof}[Proof of Theorem~\ref{thm:lipschitz-dys-a-lip-c-sm}]
        First consider the first case,
	     where $\opA$ is $L_\cA$-Lipschitz and $\opC$ is $\mu_\cC$-strongly monotone.
	     Write\[
			\cA = \cM \cap \lip{L_\cA}, \quad \cB = \cM, \quad \cC = \sm{\mu_\cC} \cap \coco{\beta_\cC} .
		\]
	    
        Let 
        \[
            O = \frac{2 - \lambda}{\alpha}, 
            \quad
            O_\beta = \frac{1}{2\beta_\cC},
            \quad
            \{ N, N' \} = \partial\srg{\cM_{\mu_\cC}} \cap \partial\srg{\cC_{\beta_\cC}}.
        \]
		Define
		\[
		    \cC' = \alpha^{-1}((2 - \lambda)\opI + \lip{R})
		\]
        so that $(2 - \lambda)\opI - \alpha \cC' = \cL_R$ satisfies an arc property and set $R$ to the minimum value that satisfies $\cC \subset \cC'$.
        Observe that $\srg{\cC'}$ is centered at $O$, which is on the right side of $O_\beta$ since $2 - \lambda > \frac{\alpha}{2\beta_\cC}$.
        Therefore, we have
        \begin{align*}
            R =     
            \alpha \overline{ON} 
            = \sqrt{(2 - \lambda)(2 - \lambda - 2(1 - \eta)\alpha\mu_\cC)}
        \end{align*}
        with $\eta = \frac{\alpha}{2\beta_\cC(2 - \lambda)}$, as depicted in the following figure.

        \begin{center}
			\def\scale{1.3}
			\begin{tikzpicture}[scale=\scale]
				\def\a{1};
				\def\b{0.7};
				\def\lam{0.4};
				\def\t{330};
				\def\k{300};
				\def\e{1.6};
				\def\m{0.35};
				\def\lipR{sqrt((2-\lam-\a*\m)^2+\a^2*(\m/\b-\m^2))};

				\coordinate (Z) at (0, 0);
				\coordinate (O) at (\a/\b/2, 0);
				\coordinate (A) at (2-\lam, 0);
				\coordinate (P) at ({\a/\b/2+\a/\b/2*cos(\t)}, {\a/\b/2*sin(\t)});
				\coordinate (L) at ({\a/\b/2+\a/\b/2*cos(\k)}, {\a/\b/2*sin(\k)});
				\coordinate (C) at ({\a*\m}, {\a*sqrt(\m/\b-\m^2)});
				\coordinate (C') at ({\a*\m}, {-\a*sqrt(\m/\b-\m^2)});
				\coordinate (M) at ({2-\lam+\lipR*cos(\t)},{\lipR*sin(\t)});
				
				\fill[fill=lightgrey] (A) circle ({\lipR});
				\begin{scope}
				    \clip ({\a*\m},{-\a/(2*\b)}) rectangle ({\a/\b},{\a/(2*\b)});
				    \fill[fill=medgrey] (O) circle ({\a/\b/2});
				\end{scope}
				

				\filldraw (O) circle[radius={\ptsize / \scale pt}];
				/\filldraw (A) circle[radius={\ptsize / \scale pt}];
				\filldraw (C) circle[radius={\ptsize / \scale pt}];
				\filldraw (C') circle[radius={\ptsize / \scale pt}];

				\draw (O) node[above right=0pt] {$O_\beta$};
				\draw (A) node[above right=0pt] {$O$};
				\draw (C) node[above left=-3pt] {$N$};
				\draw (C') node[below left=-3pt] {$N'$};
				\draw (L) node[above left=0pt] {$\srg{\cC}$};
				\draw (M) node[below right] {$\srg{\cC'}$};
				
				\pic [draw, dashed, angle radius=\scale*\cminpt*\a/\b/2] {angle=C--O--C'};
				
				\draw [<->] (-0.2,0) -- (3.2,0);
				\draw [<->] (0,-1.5) -- (0,1.5);
			\end{tikzpicture}	
		\end{center}

		By Corollary~\ref{cor:tight-coeff-dys-completion}, it suffices to show
		\[
			\abs{\zeta_{\text{DYS}}(z_A, z_B, z_{C'}; \alpha, \lambda)}
			\le \rho
		\]
		for any $z_A \in \partial\srg{\resa{\cA}}$, $z_B \in \partial\srg{\resa{\cB}}$, and
		$z_{C'} \in \partial\srg{\cC'}$. 
        We have
    	\begin{align*}
			&\left| \zeta_{\text{DYS}}(z_A, z_B, z_{C'}; \alpha, \lambda) \right|^2 \\
    		&= \left| (1 - \lambda z_A)(1 - \lambda z_B) + \lambda(2 - \lambda - \alpha z_{C'}) z_A z_B \right|^2
    		\\
    		&\le \left( \abs{1 - \lambda z_A}\abs{1 - \lambda z_B} + 
    		\lambda\abs{2 - \lambda - \alpha z_{C'}} \abs{z_A} \abs{z_B} \right)^2
			\\
			&= \left( \abs{1 - \lambda z_A}\abs{1 - \lambda z_B} + 
    		\lambda R \abs{z_A} \abs{z_B} \right)^2
			\\
			&= \left(
				\sqrt{ \abs{1 - \lambda z_A}^2\abs{1 - \lambda z_B}^2 }
				+
				\sqrt{ \left( r^{-1}\lambda R \abs{z_A}^2 \right)
				\left( r\lambda R \abs{z_B}^2 \right)
				}
			\right)^2
			\\
    		&\le \left(\abs{1 - \lambda z_A}^2 + r^{-1} \lambda R  \abs{z_A}^2\right)
    		\left(\abs{1 - \lambda z_B}^2 + r \lambda R  \abs{z_B}^2\right)
    	\end{align*}
    	for any $r>0$,
    	where first inequality follows from the triangle inequality, second equality follows from $z_{C'} \in \partial\srg{\cC'} = \Ci\left(\frac{2 - \lambda}{\alpha}, \frac{R}{\alpha}\right)$, and last inequality follows from Lemma~\ref{lem:CS}.
    	Set
		\[
			r = \frac{2 - \lambda}{R}	
		\]
		and the second factor simplifies to
    	\[
    	    \abs{1 - \lambda z_B}^2 + r \lambda R \abs{z_B}^2 = 2\lambda \left| z_B - \frac{1}{2}\right|^2 + 1 - \frac{\lambda}{2} = 1
    	\]
    	as $z_B \in \partial\srg{\resa{\cB}} = \Ci\left( \frac{1}{2}, \frac{1}{2} \right)$. Therefore,
    	\begin{align*}
    	  &\left| \zeta_{\text{DYS}}(z_A, z_B, z_{C'}; \alpha, \lambda) \right|^2 \\
    	  &\le
    	   \abs{1 - \lambda z_A}^2 + \lambda \frac{R^2}{2 - \lambda} \abs{z_A}^2
    		\\
    		&=\left( \lambda^2 + \lambda \frac{R^2}{2 - \lambda} \right)
    		\left| z_A - \frac{1}{\lambda + R^2 / (2 - \lambda)} \right|^2 + 1 - \frac{\lambda}{\lambda + R^2/(2 - \lambda)},
    	\end{align*}
		which is maximized when  $\left| z_A - \frac{1}{\lambda + R^2 / (2 - \lambda)}\right|$ is maximized.
		By the same reasoning as that of Theorem~\ref{thm:lipschitz-dys-a-lip-b-sm}, the maximum is attained at $z_A = \frac{1}{1 + \alpha^2 L_\cA^2}(1 \pm i \alpha L_\cA)$ or $z_A = \frac{1}{1 + \alpha L_\cA}$.
		(Since $\cA = \cM \cap \lip{L_\cA}$ for both the current setup and that of Theorem~\ref{thm:lipschitz-dys-a-lip-b-sm}, the illustration of $\cG(\resa{\cA})$ is identical and $Q=\frac{1}{\lambda + R^2 / (2 - \lambda)} > \frac{1}{2}=O_0$ implies $Q$ is on the right side of $O_0$).
    	In the first case where $z_A = \frac{1}{1 + \alpha L_\cA}$ is a maxima,
		\begin{align*}
			\abs{1 - \lambda z_A}^2 + \lambda \frac{R^2}{2 - \lambda} \abs{z_A}^2
			&= 1 + \frac{1}{(1 + \alpha L_\cA)^2}
			\left( -2\lambda(1 + \alpha L_\cA) + \lambda^2 + \frac{\lambda R^2}{2-\lambda} \right)
			\\
			&= 1 - \frac{2 \lambda \alpha(L_\cA + (1 - \eta)\mu_\cC)}{(1 + \alpha L_\cA)^2}
		\end{align*}
            which follows from the definition $R = \sqrt{(2 - \lambda)(2 - \lambda - 2(1 - \eta)\alpha\mu_\cC)}$.
		In the second case where $z_A = \frac{1}{1 + \alpha^2 L_\cA^2}(1 \pm i \alpha L_\cA)$ are maxima,
		\begin{align*}
			\abs{1 - \lambda z_A}^2 + \lambda \frac{R^2}{2 - \lambda} \abs{z_A}^2
			&= 1 + \frac{1}{1 + \alpha^2 L_\cA^2}\left( -2\lambda + \lambda^2 + \frac{\lambda R^2}{2-\lambda} \right)
			\\
			&= 1 - \frac{2 \lambda \alpha \mu_\cC (1 - \eta)}{1 + \alpha^2 L_\cA^2}.
		\end{align*}
		Therefore, 
		\begin{align*}
			&\left| \zeta_{\text{DYS}}(z_A, z_B, z_{C'}; \alpha, \lambda) \right|^2
			\le 1 - \min\left\{ 
				\frac{2 \lambda \alpha(L_\cA + (1 - \eta)\mu_\cC)}{(1 + \alpha L_\cA)^2}, 
				\frac{2 \lambda \alpha \mu_\cC (1 - \eta)}{1 + \alpha^2 L_\cA^2}
			\right\}.
		\end{align*}
		This completes the proof for the first statement.

		Since $\zeta_{\text{DYS}}(z_A, z_B, z_{C'}; \alpha, \lambda)$ is symmetric with respect to $z_A$ and $z_B$, the second statement follows from the same argument.
	\end{proof}	

        \subsection{Comparison to prior factors}
        \subsubsection{Contraction factors of DYS}
        
 Our contraction factors of Theorems~\ref{thm:lipschitz-dys-a-lip-sm}, \ref{thm:lipschitz-dys-a-lip-b-sm}, and \ref{thm:lipschitz-dys-a-lip-c-sm} are smaller (better) than the six prior contraction factors of  \cite[Theorem~D.6]{davis2015three}.
		However, accurate comparisons require making \emph{appropriate updates} to the prior results.

    Theorem~D.6 of \cite{davis2015three} presents six contraction factors, which are obtained by making certain substitutions to the equation (D.33) of \cite{davis2015three}. Carefully carrying out these calculations leads to the following updates.
		 In Theorem~D.6.1 of \cite{davis2015three}, the contraction factor assuming conditions of the second statement of our Theorem~\ref{thm:lipschitz-dys-a-lip-sm}
		(so $\opB$ is $\mu_\cB$-strongly monotone and $L_\cB$-Lipschitz) should be updated to
		\[
			\rho = \sqrt{1 - \frac{2\mu_\cB \alpha \lambda}{(1 + \alpha L_\cB)^2}}.
		\]
		In Theorem~D.6.3 and D.6.4 of \cite{davis2015three}, the contraction factors 
		assuming conditions of our Theorem~\ref{thm:lipschitz-dys-a-lip-b-sm}
		(so $\opA$ is $\mu_\cA$-strongly monotone and $\opB$ is $L_\cB$-Lipschitz, and vice-versa) should be updated to
		\begin{align*}
		    \rho &= \sqrt{1 -\frac{\lambda}{3} \min\left\{ 
				\frac{2\alpha\mu_\cA}{(1 + \alpha L_\cB)^2}, 
				\frac{\lambda}{1 + 2\alpha^2 L_\cB^2}\left( \frac{2-\epsilon}{\lambda} - 1 \right)
			 \right\}},\\
			 	\rho' &= \sqrt{1 - \frac{\lambda}{4}\min\left\{ 
				\frac{2\alpha \mu_\cB}{1 + 2\alpha^2 L_\cA^2}, 
				\frac{2\beta_\cC - \alpha/\epsilon}{\alpha},
				\frac{\lambda}{1 + 2\alpha^2 L_\cA^2}\left( \frac{2 - \epsilon}{\lambda} - 1 \right)
			 \right\}}
		\end{align*}
  		respectively, where $\epsilon \in (\frac{\alpha}{2\beta_\cC}, \min(1, 2 - \lambda))$.
       In Theorem~D.6.5 and~D.6.6 of \cite{davis2015three}, the contraction factors 
		assuming conditions of our Theorem~\ref{thm:lipschitz-dys-a-lip-c-sm}
		(so $\opB$ is $\mu_\cB$-strongly monotone and $\opC$ is $L_\cC$-Lipschitz, and $\opA$ is $\mu_\cA$-strongly monotone and $\opC$ is $L_\cC$-Lipschitz) should be updated to
        \begin{align*}
            \rho &= \sqrt{1 - \frac{\lambda}{4}\min\left\{ 
                \frac{2\alpha \mu_\cC(1-\eta)}{1 + 2\alpha^2 L_\cA^2}, 
                \frac{2\eta\beta_\cC - \alpha/\epsilon}{\alpha},
                \frac{\lambda}{1 + 2\alpha^2 L_\cA^2}\left( \frac{2 - \epsilon}{\lambda} - 1 \right)
             \right\}},\\
            \rho' &= \sqrt{1 - \frac{2\alpha\lambda\mu_\cC(1 - \eta)}{(1 + \alpha L_\cB)^2}},
        \end{align*}
        		where $\eta \in (0, 1)$ is a constant satisfying $\eta > \frac{\alpha}{2\beta_\cC \epsilon}$.
		($\epsilon$ is as defined previously).

        \begin{figure}
            \centering
            \includegraphics[width=0.9\textwidth]{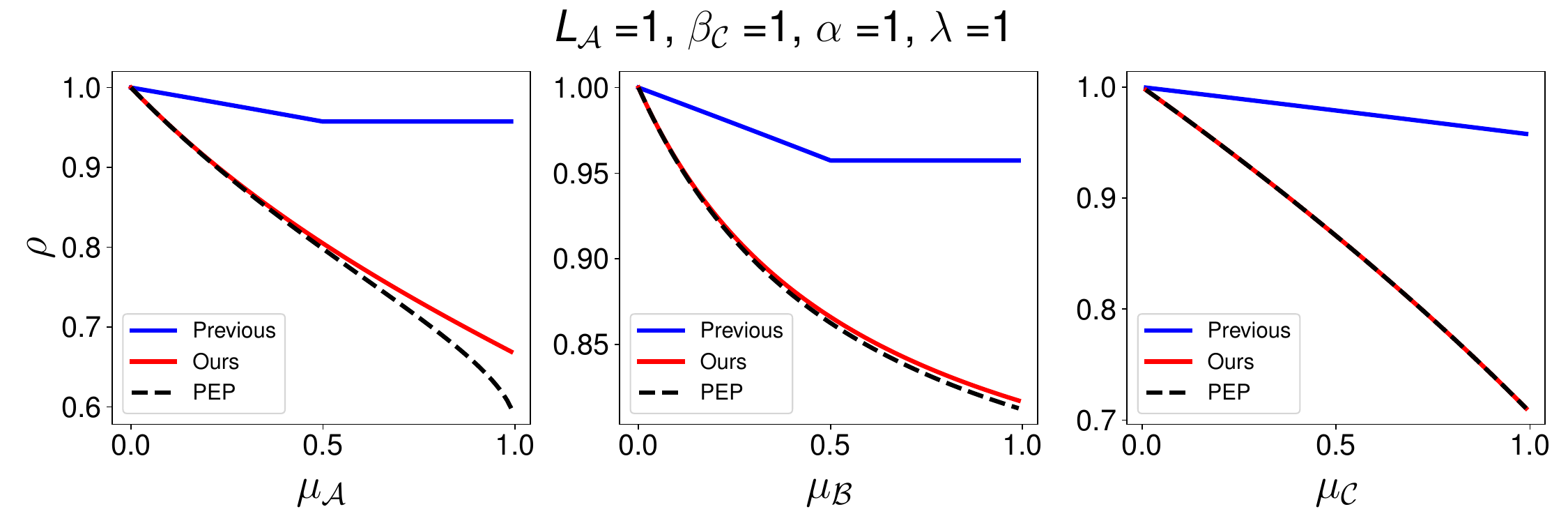}
            \includegraphics[width=0.9\textwidth]{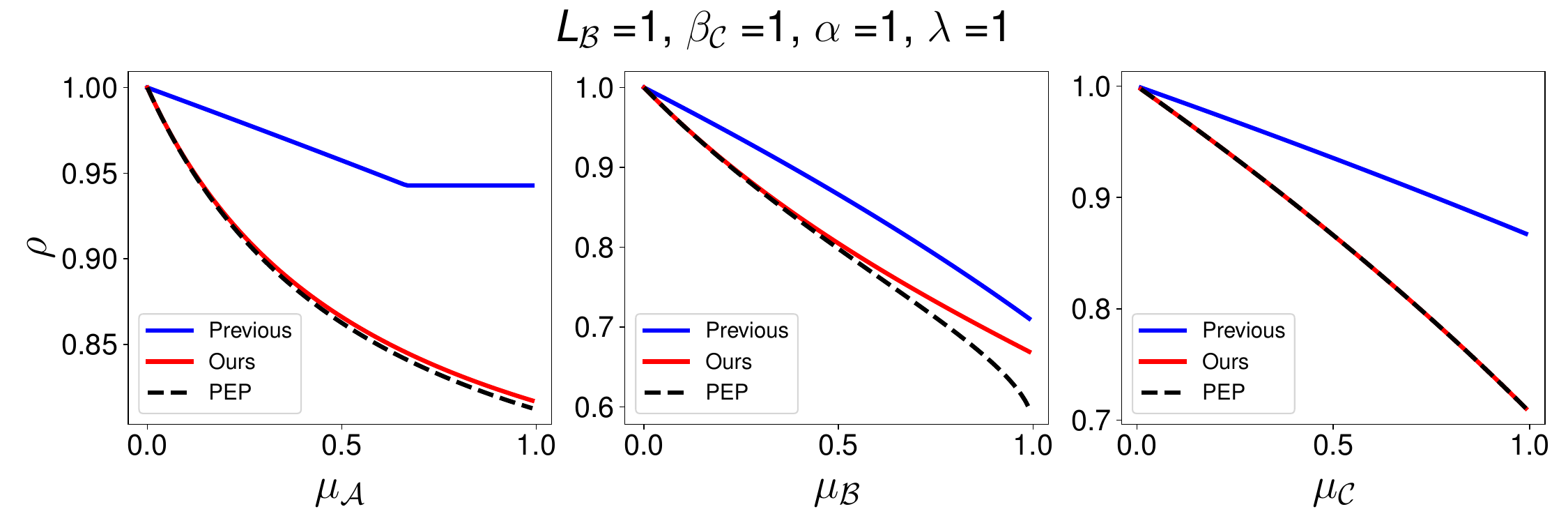}
            \caption{DYS contraction factors obtained by the previous work \cite{davis2015three} (blue), our work (red), and the tight factor calculated via PEP \cite{ryu_operator_2020} (black). 
            For every case, we set $\alpha = \lambda = 1$ and operators in $\cC$ to be $\beta_\cC = 1$-cocoercive.
            We set $\cA$ and $\cB$ as $1$-Lipschitz-monotone class in the plots above and below, respectively. 
            We further impose the strong monotonicity for $\cA$, $\cB$, and $\cC$ in each case and plot the contraction factors accordingly.
            }
            \label{fig:comparison-to-previous-rates}
        \end{figure}
        By simple calculation, we could verify that our six contraction factors are strictly better than those updated prior rates, respectively. The detailed calculations comparing our rates against the prior rates of \cite{davis2015three} are carried out in the arXiv version of this paper \cite[Section~3.2]{lee2022convergence}. In this section, we want to focus on the numerical results of this comparison with the worst-case rate obtained by the performance estimation problem (PEP) framework \cite{drori2014performance,taylor2017smooth,ryu_operator_2020}.
        For a fair comparison with prior rates, we optimize the choice of $\epsilon$ and $\eta$ to minimize the factors above and compare with our results. As depicted in Figure~\ref{fig:comparison-to-previous-rates}, the contraction factors we derived are much closer to the tight contraction factors evaluated by the PEP approach.
        In this sense, our contraction factors are a better proxy of the tight contraction factors than the earlier results in \cite{davis2015three}.
        We further investigate the advantages of our rates in assessing the computational cost, especially in an ill-conditioned case where the coefficient of the strong monotonicity $\mu$ imposed to $\cA$, $\cB$, or $\cC$ is infinitesimal.
        For the six cases addressed in Theorems~\ref{thm:lipschitz-dys-a-lip-sm}, \ref{thm:lipschitz-dys-a-lip-b-sm}, and \ref{thm:lipschitz-dys-a-lip-c-sm}, one can simply plug in $\mu = 0$ and confirm that the contraction factor $\rho = 1$ in such a case ($\mu$ is one of the $\mu_\cA$, $\mu_\cB$, or $\mu_\cC$ appearing in each case).
        The same applies to the contraction rates provided in the earlier work. 
        Therefore, one can approximate a contraction factor $\rho$ via the first-order Taylor series expansion: $\rho \approx 1 +  \mu (d\rho / d\mu |_{\mu = 0})$.
        If we want to evaluate the computational cost to reach a predetermined degree of error, we should consider 
        \[
            -\frac{1}{\log \rho}
            \approx -\frac{1}{\log \left( 1 +  \mu \left( \frac{d\rho}{d\mu} \vert_{\mu = 0} \right) \right)}
             \approx  -\frac{1}{\mu \left( \frac{d\rho}{d\mu} \vert_{\mu = 0} \right)}
        \]
        as it is proportional to the upper bound of the total number of iterations.
        Therefore, denoting the contraction factors in our and the previous work \cite{davis2015three} by $\rho_{\mathrm{ours}}$ and $\rho_{\mathrm{prev}}$ respectively, $\frac{d \rho_{\mathrm{ours}}}{d \mu} / \frac{d \rho_{\mathrm{prev}}}{d \mu} \rvert_{\mu = 0}$ would be an appropriate proxy to assess which of two are better; larger the value is, better our factor is.
        As presented in Table~\ref{tab:comparison-taylor-1d-v2}, this value is strictly larger than $1$ in every case, meaning that our factors are always better than the prior factors. 
        Indeed, these discrepancies are even higher than $4$ in some cases, supporting significant improvement of our results in the sense we are aware of.



        \begin{table}[]
        \renewcommand{\arraystretch}{1.5}
       \centering
       \caption{Evaluation of $\frac{d \rho_{\mathrm{ours}}}{d \mu} / \frac{d \rho_{\mathrm{prev}}}{d \mu} \rvert_{\mu = 0}$. 
       Each row corresponds to one of the six settings addressed in Theorems~\ref{thm:lipschitz-dys-a-lip-sm}, \ref{thm:lipschitz-dys-a-lip-b-sm}, and \ref{thm:lipschitz-dys-a-lip-c-sm} following the same order as in the previous theorems.
       To clarify, $\rho_{\mathrm{ours}}$ and $\rho_{\mathrm{prev}}$ are the contraction factors presented in our work and the previous work \cite{davis2015three}, respectively, and $\mu$ represents $\mu_\cA$, $\mu_\cB$, or $\mu_\cC$ in each row.
       }
       \begin{tabular}{cccc}
       \toprule
    $\cA$ & $\cB$ & $\cC$ & $\frac{d \rho_{\mathrm{ours}}}{d \mu} / \frac{d \rho_{\mathrm{prev}}}{d \mu} \rvert_{\mu = 0}$  \\ 
       \midrule
       $\sm{\mu_\cA} \cap \lip{L_\cA}$ & $\cM$ & $\coco{\beta_\cC}$ & $3(1 + \alpha L_\cA)^2 /( 1 + \alpha^2 L_\cA^2)$ \\
       $\cM$ & $\sm{\mu_\cB} \cap \lip{L_\cB}$ & $\coco{\beta_\cC}$ & $(1 + \alpha L_\cB)^2 /( 1 + \alpha^2 L_\cB^2)$ \\
       $\lip{L_\cA}$ & $\sm{\mu_\cB}$ & $\coco{\beta_\cC}$ & $4 (1 + 2\alpha^2 L_\cA^2) / (1 + \alpha^2 L_\cA^2)$\\ 
       $\sm{\mu_\cA}$ & $\lip{L_\cB}$ & $\coco{\beta_\cC}$ & $3 (1 + \alpha L_\cB)^2 / (1 + \alpha^2 L_\cB^2)$ \\ 
       $\lip{L_\cA}$& $\cM$ & $\sm{\mu_\cC} \cap \coco{\beta_\cC}$ & $4 (1 + 2 \alpha^2 L_\cA^2) /( 1 + \alpha^2 L_\cA^2)$ \\ 
       $\cM$ & $\lip{L_\cB}$ & $\sm{\mu_\cC} \cap \coco{\beta_\cC}$ & $(1 + \alpha L_\cB)^2 /( 1 + \alpha^2 L_\cB^2)$\\ 
    \bottomrule
       \end{tabular}
   \label{tab:comparison-taylor-1d-v2}
        \end{table}

    \subsubsection{Contraction factors of FBS}
    As DYS generalizes various two-operator splittings such as FBS and DRS, it is natural to ask what contraction rates are induced by our results in two-operator splitting setups and how strong they are compared to prior results.
    In particular, plugging $\opB = 0$ into $\opT_{\opA, \opB, \opC, \alpha, \lambda}$ gives $\opT_{\opA, 0, \opC, \alpha, \lambda} = (1 - \lambda) \opI + \lambda \resa{\opA} (\opI - \alpha \opC)$, an averaged version of the FBS operator.
    We consider a case where $\opA$ is $\mu_\cA$-strongly monotone and $\opC$ is $\beta_\cC$-cocoercive, which is also studied by Guo \cite{guo2021onthelinear}.
    We first establish the contraction factor of FBS implied by Theorem~\ref{thm:lipschitz-dys-a-lip-b-sm} and compare it with the rate provided in \cite{guo2021onthelinear}.
    
    Recall that the second instance of Theorem~\ref{thm:lipschitz-dys-a-lip-b-sm}; since $\opB = 0$ is $L_\cB$-Lipschitz for any $L_\cB > 0$, the contraction factor stated in Theorem~\ref{thm:lipschitz-dys-a-lip-b-sm} is valid for any $L_\cB > 0$ in this case. 
    Therefore, we can guarantee that $\opT_{\opA, 0, \opC, \alpha, \lambda}$ has a contraction factor smaller than or equal to the infimum of those factors. 
    Therefore, $\opT_{\opA, 0, \opC, \alpha, \lambda}$ is $\rho_{\mathrm{ours}}$-contractive where
    \begin{align*}
        \rho_{\mathrm{ours}} = \sqrt{1 - \frac{2 \alpha \mu_{\cA}\lambda(2-\lambda)}{2-\lambda+2\alpha \mu_{\cA}}}.
    \end{align*}
    The prior rate provided in Theorem~3.1 of \cite{guo2021onthelinear} is translated 
    \[
        \rho_{\mathrm{prev}} = \begin{cases}
            1 - \lambda \left( 1 - \frac{2\beta_\cC - \alpha + \alpha^2 \mu_\cA}{2 \alpha \beta_\cC \mu_\cA + 2\beta_\cC - \alpha} \right) 
            & \text{if } 0 < \lambda \leq \frac{8 \alpha \beta_\cC^2 \mu_\cA + 4\beta_\cC(2\beta_\cC - \alpha)}{8 \alpha \beta_\cC^2 \mu_\cA + (2\beta_\cC + \alpha)(2\beta_\cC - \alpha)}, \\
            \left( 1 + \frac{\alpha}{2\beta_\cC} \right) \lambda - 1 & \text{if } \frac{8 \alpha \beta_\cC^2 \mu_\cA + 4\beta_\cC(2\beta_\cC - \alpha)}{8 \alpha \beta_\cC^2 \mu_\cA + (2\beta_\cC + \alpha)(2\beta_\cC - \alpha)} < \lambda < 2.
        \end{cases}
    \]
    under $0 < \alpha < \min(2 \beta_\cC, 2\beta_\cC (2 - \lambda) / \lambda)$.
    \begin{figure}
            \centering
            \includegraphics[width=0.9\textwidth]{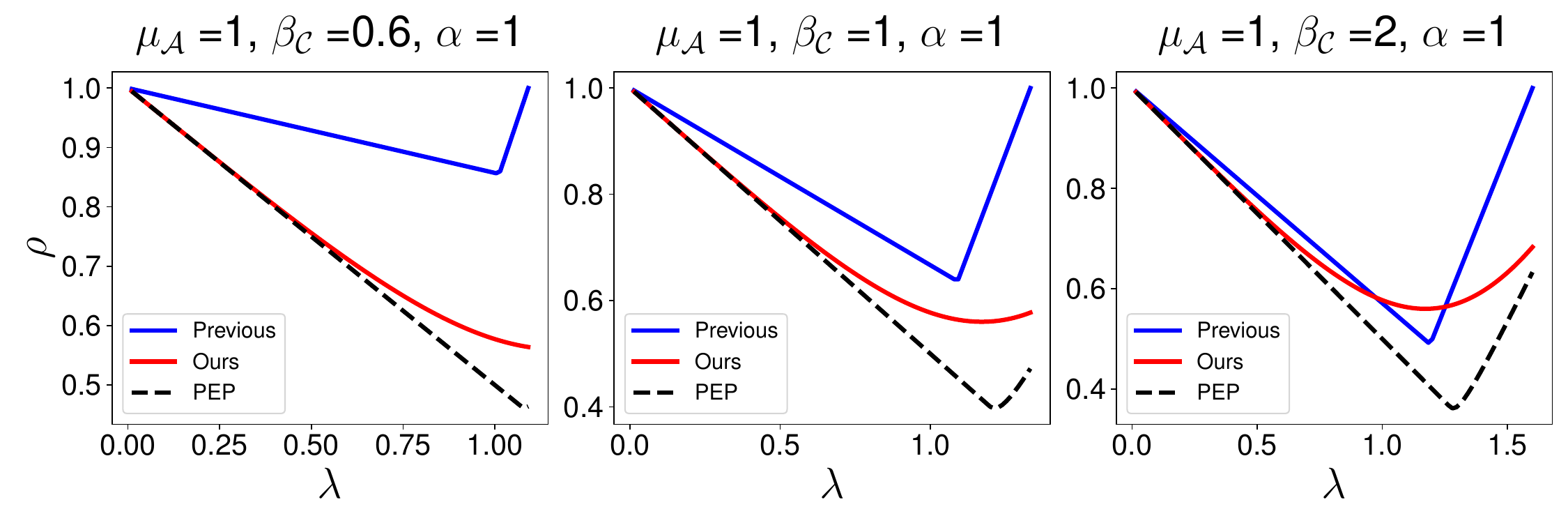}
            \includegraphics[width=0.9\textwidth]{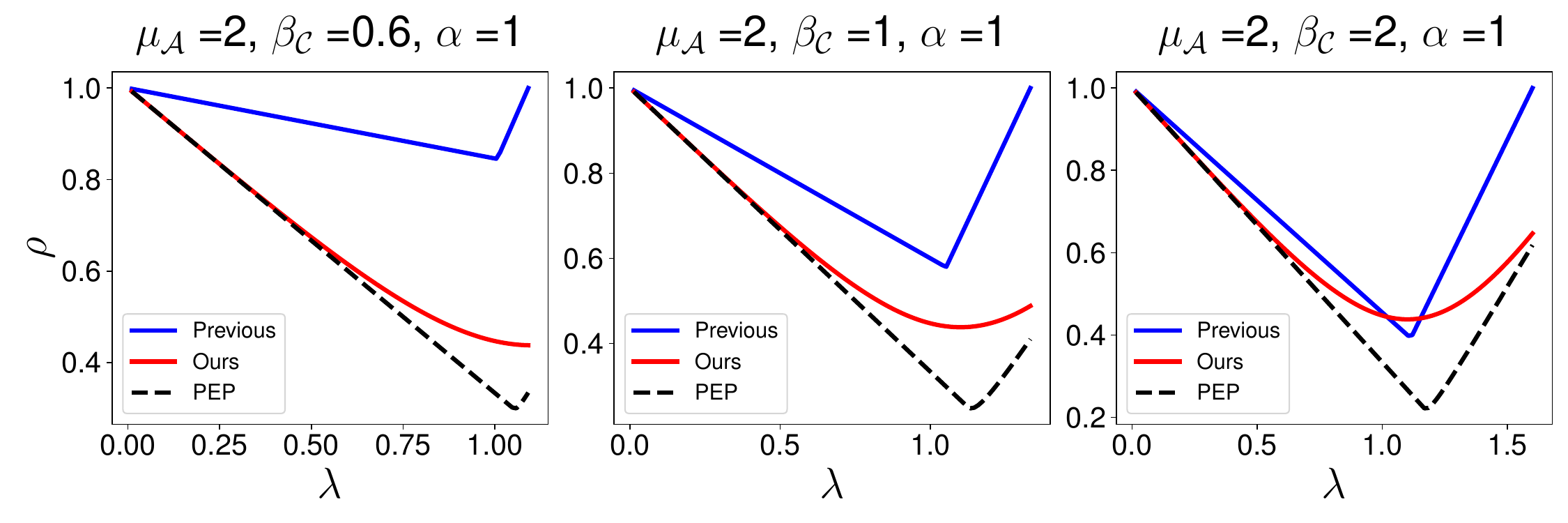}
            \caption{FBS contraction factors obtained by the previous work \cite{guo2021onthelinear} (blue), our work (red), and the tight factor calculated via PEP \cite{ryu_operator_2020} (black). 
            With six different choices of $\mu_\cA$, $\beta_\cC$, and $\alpha$, we evaluated each contraction factor over the range of $\lambda$ where Theorem~3.1 of \cite{guo2021onthelinear} is valid.}
        \label{fig:comparison-to-previous-rates-fbs}
    \end{figure}
    As shown in Figure~\ref{fig:comparison-to-previous-rates-fbs}, our rates are generally closer to the tight rates than the prior rates are, except for narrow ranges of $\lambda$.

    \section{Averagedness of DYS}
	\label{s:dys_avg}
	In this section, we present a novel averagedness result of the DYS operator that does not require cocoercivity of $\opC$.

	
	\begin{theorem}\label{thm:dys-avg}
		Let $\cH$ be a real Hilbert space. Let $\opA \colon \cH \rightrightarrows \cH$, $\opB \colon \cH \rightrightarrows \cH$, and $\opC \colon \cH \rightrightarrows \cH$ be monotone operators.
		If $\opA$ is $\mu_\cA$-strongly monotone and $\opC$ is $L_\cC$-Lipschitz
		(so $\opC$ need not be cocoercive) with $\mu_\cA$, $L_\cC \in (0, \infty)$, and $\alpha \in ( 0, 2\mu_\cA/L_\cC^2 )$,
		then $\opT_{\opA, \opB, \opC, \alpha, 1}$ (so $\lambda=1$) is $\frac{2}{4 - \alpha L_\cC^2 / \mu_\cA}$-averaged.
		Symmetrically, if $\opB$ is  $\mu_\cB$-strongly monotone and $\opC$ is $L_\cC$-Lipschitz with $\mu_\cB, L_\cC\in (0,\infty)$, and $\alpha \in ( 0, 2\mu_\cB/L_\cC^2 )$, then $\opT_{\opA, \opB, \opC, \alpha, 1}$ is $\frac{2}{4 - \alpha L_\cC^2 / \mu_\cB}$-averaged.
	\end{theorem}

	\begin{proof}
		Consider the first statement, where $\opA$ is $\mu_\cA$-strongly monotone and $\opC$ is $L_\cC$-Lipschitz. Write $\cA = \sm{\mu_\cA}$, $\cB = \cM$, and $\cC = \cM \cap \lip{L_\cC}$.
		Let $\theta = \frac{2}{4 - \alpha L_\cC^2 / \mu_{\cA}}$.
		Then the statement is equivalent to saying that $\cG(\opT_{\cA, \cB, \cC, \alpha, 1})\subseteq\cG(\cN_\theta)$.

		We now follow a similar argument in the proof of Theorem~\ref{thm:lipschitz-dys-a-lip-c-sm}. 
            Define
            \[
                O = \frac{2 - \theta^{-1}}{\alpha}, \quad O_L = 0,\quad \{ N, N' \} = \partial \srg{\cM} \cap \partial \srg{\lip{L_\cC}},
            \]
            and
		\[
			\cC' = \alpha^{-1}((2 - \theta^{-1})\opI + \cL_R),
		\]
		to have $(2 - \theta^{-1})\opI - \alpha \cC' = \cL_R$ satisfies an arc property and set $R$ to the minimum value that renders $\cC \subset \cC'$.
        Note that $\srg{\cC'}$ is centered at $O$, which is on the right side of $O_L$ since $2 - \theta^{-1} > 0$.
        Therefore, we have $R = \alpha \overline{ON} = \sqrt{(2 - \theta^{-1})^2 + \alpha^2 L_{\cC}^2}$ as depicted in the following figure.
        \begin{center}
		\def\scale{1.6}
		\begin{tikzpicture}[scale=\scale]
			\def\a{0.5};
			\def\L{1};
			\def\m{0.7};
			\def\t{305};
			\def\k{320};
			\def\e{1.6};
                \def\lam{1};
                \def\radone{(\L^2  / (2  * \lam * \m))};

			\coordinate (Z) at (0, 0);
			\coordinate (O) at (0, 0);
			\coordinate (O') at ({\L^2  / (2  * \lam * \m)}, {0});
			\coordinate (C) at (0, {\L});
			\coordinate (C') at (0, {-\L});
			
			\begin{scope}
			    \fill[fill=lightgrey] (O') circle ({sqrt(\radone^2 + \L^2)});
			\end{scope}
			
			\begin{scope}
			    \clip (C') rectangle (\L, \L);
			    \fill[fill=medgrey] (O) circle ({\L});
			\end{scope}
		  
		  
		    \pic[draw, dashed, angle radius={\cminpt*\scale*\L}] {angle=C--O--C'};

			\filldraw (O) circle[radius={\ptsize / \scale pt}];
			\filldraw (O') circle[radius={\ptsize / \scale pt}];
 			\filldraw (C) circle[radius={\ptsize / \scale pt}];
 			\filldraw (C') circle[radius={\ptsize / \scale pt}];

 			\draw (O) node[below left] {$O_L$};
 			\draw (O') node[below =0pt and -5pt] {$O$};
 			\draw (C) node[above left=0pt and 0pt] {$N$};
 			\draw (C') node[below left=0pt and 0pt] {$N'$};
			
			\draw ({\L*cos(\t)},{\L*sin(\t)}) node[above left] {$\srg{\cC}$};
			\draw ({\L^2/(2*\m) + sqrt(\L^2+\L^4/\m^2/4)*cos(\k)},{sqrt(\L^2+\L^4/\m^2/4)*sin(\k)}) node[below right] {$\srg{\cC'}$};

			\draw [<->] (-1.2,0) -- (2.2,0);
			\draw [<->] (0,-1.4) -- (0,1.4);
		\end{tikzpicture}	
    	\end{center}
 
		Continuing on,
		\begin{align*}
		   	\abs{\srg{\opT_{\cA, \cB, \cC, \alpha, 1}} - (1 - \theta)} 
			&\subseteq
			\abs{\srg{\opT_{\cA, \cB, \cC', \alpha, 1}} - (1 - \theta)} \\
			&=
			\abs{\cZ^{\text{DYS}}_{\cA, \cB, \cC', \alpha, 1} - (1 - \theta)},
		\end{align*}
		where the equality is a consequence of Theorem~\ref{thm:tight-dys-main} with $\lambda = 1$ and $s = 1 - \theta$.
		Now, it suffices to show 
			$\sup_{z \in \cZ^{\text{DYS}}_{\cA, \cB, \cC', \alpha, 1}} \abs{z - (1 - \theta)} \le \theta$,
		or equivalently, by Lemma~\ref{lem:max-mod},
		\[
		    \abs{\zeta_{\text{DYS}}(z_A, z_B, z_{C'}; \alpha, 1)-(1-\theta)} \le \theta
		\]
		for any $z_A \in \partial\srg{\resa{\cA}}$, $z_B \in \partial\srg{\resa{\cB}},$ and $z_{C'} \in \partial\srg{\cC'}$.
		
        With a similar argument as that for Theorem~\ref{thm:lipschitz-dys-a-lip-sm}, we have
		\begin{align*}
			&\abs{\zeta_{\text{DYS}}(z_A, z_B, z_{C'}; \alpha, 1)-(1-\theta)}^2\\
			&=\abs{\theta^{-1}\para{z_A-\theta}\para{z_B-\theta} +(2-\theta^{-1} -\alpha z_{C'}) z_Az_B }^2\\
			&\le \para{\theta^{-1}\abs{z_A-\theta}\abs{z_B-\theta}+\abs{2-\theta^{-1}-\alpha z_{C'}}\abs{z_A}\abs{z_B}}^2
                \\
			\qquad\qquad\quad &= \bigg(\sqrt{\theta^{-1} \abs{z_A-\theta}^2 \cdot \theta^{-1} \abs{z_B-\theta}^2} \\
                &\quad\quad + \sqrt{\abs{2-\theta^{-1}-\alpha z_{C'}}^2\para{2-\theta^{-1}}^{-1}\abs{z_A}^2 \cdot \para{2-\theta^{-1}}\abs{z_B}^2} \bigg)^2\\
			&\le \para{\theta^{-1} \abs{z_A-\theta}^2+\abs{2-\theta^{-1}-\alpha z_{C'}}^2\para{2-\theta^{-1}}^{-1}\abs{z_A}^2}\\
			&\quad \cdot \para{\theta^{-1} \abs{z_B-\theta}^2+\para{2-\theta^{-1}}\abs{z_B}^2},
		\end{align*}
		where the first inequality follows from the triangle inequality and the second is from Lemma~\ref{lem:CS}.
		
		The second factor simplifies to
        \[
		    \theta^{-1} \abs{z_B-\theta}^2+\para{2-\theta^{-1}}\abs{z_B}^2
		    = 2\left| z_B - \frac{1}{2} \right|^2 + \theta - \frac{1}{2} = \theta
		\]
		since $z_B \in \partial\srg{\resa{\cB}} = \Ci\left(\frac{1}{2}, \frac{1}{2}\right)$.
		
		The first factor simplifies to
		\begin{align*}
		    &\theta^{-1} \abs{z_A-\theta}^2+\abs{2-\theta^{-1}-\alpha z_{C'}}^2\para{2-\theta^{-1}}^{-1}\abs{z_A}^2\\
		    &\stackrel{\text{(i)}}{=} \theta^{-1} \abs{z_A-\theta}^2+R^2(2 - \theta^{-1})^{-1}\abs{z_A}^2 \\
		    &= \left(\theta^{-1} + R^2(2 - \theta^{-1})^{-1}\right)\left| z_A - \frac{1}{\theta^{-1} + R^2(2 - \theta^{-1})^{-1}} \right|^2 + \theta - \frac{1}{\theta^{-1} + R^2(2 - \theta^{-1})^{-1}}
		    \\
		    &\stackrel{\text{(ii)}}{=} 2(1 + \alpha\mu_\cA)\left| z_A - \frac{1}{2(1 + \alpha\mu_\cA)} \right|^2 + \theta - \frac{1}{2(1 + \alpha\mu_\cA)}\\
		    &\stackrel{\text{(iii)}}{=}\theta,
		\end{align*}
        where we used $z_{C'} \in \partial\srg{\cC'} = \Ci\left( \frac{2 - \theta^{-1}}{\alpha}, \frac{R}{\alpha} \right)$ in (i), $\theta^{-1} + R^2(2 - \theta^{-1})^{-1} = 2(1 + \alpha\mu_\cA)$ in (ii), and 
		$z_A \in \partial\srg{\resa{\cA}} = \Ci\left( \frac{1}{2(1 + \alpha\mu_\cA)}, \frac{1}{2(1 + \alpha\mu_\cA)} \right)$ in (iii).
	    Thus, we conclude 
	    \[
	        \abs{\zeta_{\text{DYS}}(z_A, z_B, z_{C'}; \alpha, 1)-(1-\theta)}^2 \le \theta^2.
	    \]
	    
	 	Since $\zeta_{\text{DYS}}(z_A, z_B, z_{C'}; \alpha, \lambda)$ is symmetric with respect to $z_A$ and $z_B$, the second statement follows from the same argument.
	\end{proof}

    The averagedness coefficient of Theorem~\ref{thm:dys-avg} is, in fact, tight in the sense that it cannot be reduced (improved) without further assumptions. This can be proved by showing that there is a $z\in \cZ^{\text{DYS}}_{\cA, \cB, \cC, \alpha, 1}$ such that $z\ne 1$ and $z\in \partial \cG(\cN_\theta)$. However, we omit this proof for the sake of conciseness.

	\section{Conclusion}
In this work, we extended the SRG theory to accommodate the DYS operator and used this theory to analyze the convergence of DYS  iterations. In our view, the SRG theory of Section~\ref{s:srg_theory} opens the door to many avenues of future work on splitting methods, and the analyses of Section~\ref{s:dys_cont} serve as demonstrations of such possibilities. 
One possible future direction is to further refine the analysis on the complex plane and obtain tighter contraction factors, as the contraction factors $\rho$ presented in Section~\ref{s:dys_cont} are not tight in the sense that 
\[
    \rho > \sup_{\substack{\opT \in \opT_{\cA, \cB, \cC, \alpha, \lambda} \\ x, y \in \dom\opT,\, x \ne y}}
    \frac{\norm{\opT x - \opT y}}{\norm{x - y}}.
\]
Another direction is to characterize optimal or near-optimal choices of parameters $\alpha$ and $\lambda$ given a set of assumptions on the operators.

\bibliographystyle{siamplain}
\bibliography{DYS}
\end{document}

%% file: macros.tex
\usepackage{cite}
\usepackage{lipsum}
\usepackage{amsfonts}
\usepackage{graphicx}
\usepackage{epstopdf}
\usepackage{algorithmic}
\usepackage{tikz}
	\usetikzlibrary{fit, arrows, shapes, positioning, shadows, matrix, calc, positioning, shadows, angles, intersections}
	
		\tikzset{cross/.style={cross out, draw=black, minimum size=2*(#1-\pgflinewidth), inner sep=0pt, outer sep=0pt},
	cross/.default={4pt}}
	
	\usetikzlibrary{decorations.pathreplacing}
	\tikzstyle std=[line width=0.7pt]   
	\tikzstyle stdthin=[line width=0.3pt]   
	\tikzstyle stdthick=[line width=1.0pt]   
	
	\tikzstyle fwd=[line width=0.7pt, ->]   
	\tikzstyle fwdthin=[line width=0.3pt, ->]   
	\tikzstyle fwdthick=[line width=1.0pt, ->]   
	\tikzstyle fwddash=[line width=0.7pt, dashed, ->]   
	
	\tikzstyle bwd=[double, line width=0.3pt, ->]  
	\tikzstyle refl=[double,  dashed, line width=0.2pt, ->]    
	
	\tikzstyle{every node}=[font=\small] 
	
	\tikzset{
		>=stealth', 
		invisible/.style={opacity=0}, 
		alt/.code args={<#1>#2#3}{\alt<#1>{\pgfkeysalso{#2}}{\pgfkeysalso{#3}}}, 
		visible on/.style={alt=#1{}{invisible}}, 
		smallnode/.style={circle, fill=black, thick, inner sep=1pt, minimum size=1.5pt}, 
		punkt/.style={
				rectangle,
				rounded corners,
				draw=black, very thick,
				text width=5.5em,
				minimum height=2em,
				text centered},
		punkt_big/.style={
				rectangle,
				rounded corners,
				draw=black, very thick,
				text width=7em,
				minimum height=2em,
				text centered},
	}
\ifpdf
  \DeclareGraphicsExtensions{.eps,.pdf,.png,.jpg}
\else
  \DeclareGraphicsExtensions{.eps}
\fi

\newsiamremark{remark}{Remark}
\newsiamremark{hypothesis}{Hypothesis}
\crefname{hypothesis}{Hypothesis}{Hypotheses}
\newsiamthm{claim}{Claim}

\headers{Convergence Analysis of DYS via SRG}{Lee, Yi, and Ryu}

\title{Convergence Analyses of Davis--Yin Splitting via Scaled Relative Graphs}

\author{Jongmin Lee$^{\dagger\ddagger}$
\and Soheun Yi$^{\S\ddagger}$
\and Ernest K.\ Ryu$^{\dagger}$
}

\usepackage{amsopn}

\usepackage{amsmath, amsfonts, mathtools, amssymb}

\usepackage[color,final]{showkeys}

\newcommand{\cut}[1]{{}}

\makeatletter
\newcommand*{\rom}[1]{\expandafter\@slowromancap\romannumeral #1@}
\makeatother

\newcommand{\cA}{{\mathcal{A}}}
\newcommand{\cB}{{\mathcal{B}}}
\newcommand{\cC}{{\mathcal{C}}}

\newcommand{\cG}{{\mathcal{G}}}
\newcommand{\cH}{{\mathcal{H}}}

\newcommand{\cL}{{\mathcal{L}}}
\newcommand{\cM}{{\mathcal{M}}}
\newcommand{\cN}{{\mathcal{N}}}

\newcommand{\cZ}{{\mathcal{Z}}}

 \usepackage[bb=boondox]{mathalfa}

\usepackage{xparse}
\DeclareFontFamily{U}{ntxmia}{}
\DeclareFontShape{U}{ntxmia}{m}{it}{<-> ntxmia }{}
\DeclareFontShape{U}{ntxmia}{b}{it}{<-> ntxbmia }{}
\DeclareSymbolFont{lettersA}{U}{ntxmia}{m}{it}
\SetSymbolFont{lettersA}{bold}{U}{ntxmia}{b}{it}

\ExplSyntaxOn
\NewDocumentCommand{\varmathbb}{m}
 {
  \tl_map_inline:nn { #1 }
  {
    \use:c { varbb##1 }
  }
 }
\tl_map_inline:nn { ABCDEFGHIJKLMNOPQRSTUVWXYZ }
 {
  \exp_args:Nc \DeclareMathSymbol{varbb#1}{\mathord}{lettersA}{\int_eval:n { `#1+67 }}
 }
\exp_args:Nc \DeclareMathSymbol{varbbk}{\mathord}{lettersA}{169}
\ExplSyntaxOff

\DeclareMathOperator*{\argmax}{argmax}

\DeclarePairedDelimiter{\dotp}{\langle}{\rangle}

\renewcommand{\Re}{\operatorname{Re}} 	
\renewcommand{\Im}{\operatorname{Im}}	

\newcommand{\dom}{\mathrm{dom}\,} 

\newcommand{\reals}{\mathbb{R}}
\newcommand{\complex}{\mathbb{C}}
\newcommand{\binfty}{{\boldsymbol \infty}}

\newcommand{\ecomplex}{\overline{\mathbb{C}}}

\makeatletter
\DeclareFontFamily{U}{tipa}{}
\DeclareFontShape{U}{tipa}{m}{n}{<->tipa10}{}
\newcommand{\arc@char}{{\usefont{U}{tipa}{m}{n}\symbol{62}}}%

\newcommand{\arc}[1]{\mathpalette\arc@arc{#1}}

\newcommand{\arc@arc}[2]{%
  \sbox0{$\m@th#1#2$}%
  \vbox{
    \hbox{\resizebox{\wd0}{\height}{\arc@char}}
    \nointerlineskip
    \box0
  }%
}
\makeatother

\newcommand{\rarc}{\mathrm{Arc}^+}
\newcommand{\larc}{\mathrm{Arc}^-}

\definecolor{lightgrey}{gray}{0.8}
\definecolor{medgrey}{gray}{0.6}
\definecolor{darkgrey}{gray}{0.4}

\newcommand{\opA}{{\varmathbb{A}}}
\newcommand{\opB}{{\varmathbb{B}}}
\newcommand{\opC}{{\varmathbb{C}}}
\newcommand{\opD}{{\varmathbb{D}}}

\newcommand{\opI}{{\varmathbb{I}}}
\newcommand{\opJ}{{\varmathbb{J}}}

\newcommand{\opT}{{\varmathbb{T}}}

\newcommand{\Di}{{\mathrm{Disk}}}
\newcommand{\Ci}{{\mathrm{Circ}}}

\DeclarePairedDelimiter{\norm}{\lVert}{\rVert}
\DeclarePairedDelimiter{\abs}{\lvert}{\rvert}
\DeclarePairedDelimiter{\babs}{\bigg\lvert}{\bigg\rvert}
\DeclarePairedDelimiter{\pths}{\bigg{(}}{\bigg{)}}

\makeatletter
\newlength{\doublefracgap}
\DeclareRobustCommand{\doublefrac}[2]{%
  \mathinner{\mathpalette\doublefrac@{{#1}{#2}}}%
}
\newcommand{\doublefrac@}[2]{\doublefrac@@#1#2}
\newcommand{\doublefrac@@}[3]{%
  \ooalign{%
    \raisebox{\doublefracgap}{$\m@th#1\frac{#2}{\phantom{#3}}$}\cr
    \raisebox{-\doublefracgap}{$\m@th#1\frac{\phantom{#2}}{#3}$}\cr
  }%
}

\makeatother

\newcommand{\ptsize}{1}
\newcommand{\cminpt}{28.3465}
 
\newcommand{\para}[1]{\left({#1}\right)}

\newcommand{\R}{\mathbb{R}}
\newcommand{\resa}[1]{\opJ_{\alpha #1}}
\newcommand{\srg}[1]{\cG\left(#1\right)}

\newcommand{\lip}[1]{\cL_{#1}}
\newcommand{\coco}[1]{\cC_{#1}}
\newcommand{\sm}[1]{\cM_{#1}}

\newcommand{\conj}[1]{\overline{#1}}
\newcommand{\diffang}[4]{%
  	\begingroup 
		\def\firstop{#1}%
		\def\secondop{#2}%
		\def\identity{\opI}%
		\ifx\firstop\identity
		\def\firstop{}%
		\fi
		\ifx\secondop\identity
		\def\secondop{}%
		\fi
		\angle (\firstop #3 - \firstop #4, \secondop #3 - \secondop #4)
  	\endgroup
}
\newcommand{\diffnormfrac}[4]{%
  	\begingroup 
		\def\firstop{#1}%
		\def\secondop{#2}%
		\def\identity{\opI}%
		\ifx\firstop\identity
		\def\firstop{}%
		\fi
		\ifx\secondop\identity
		\def\secondop{}%
		\fi
		\frac{\norm{\firstop #3 - \firstop #4}}{\norm{\secondop #3 - \secondop #4}}
  	\endgroup
}
\newcommand{\srgval}[4]{\diffnormfrac{#1}{#2}{#3}{#4}\exp[i\diffang{#1}{#2}{#3}{#4}]}